\documentclass{svjour2}                    
\smartqed  
\usepackage{graphicx}
\usepackage{amssymb, latexsym, amsmath}
\usepackage{enumerate}
\usepackage{tikz}

\newtheorem{fact}[theorem]{Fact}

\newcommand{\prob}{\mathbb{P}}
\newcommand{\expect}{\mathbb{E}}
\begin{document}

\title{
A sharp threshold for a modified bootstrap percolation with recovery}

\titlerunning{Bootstrap percolation with recovery}

\author{Tom Coker \and
	Karen Gunderson}
	
\authorrunning{T. Coker and K. Gunderson}

\institute{
	T. Coker \at
	Department of Pure Mathematics and Mathematical Statistics, 
	University of Cambridge, UK
	 \and 
	 K. Gunderson \at
	 Department of Mathematical Sciences,
	University of Memphis, Memphis, TN 38152, USA\\
	  \emph{Present address: }
	 Heilbronn Institute for Mathematical Research, School of Mathematics, University of Bristol, Bristol BS8 1TW, UK.\\
	 \email{karen.gunderson@bristol.ac.uk}}
	 
\date{Received: date / Accepted: date}

\maketitle

\begin{abstract}
Bootstrap percolation is a type of cellular automaton on graphs, introduced as a simple model of the dynamics of ferromagnetism.  Vertices in a graph can be in one of two states: `healthy' or `infected' and from an initial configuration of states, healthy vertices become infected by local rules. While the usual bootstrap processes are monotone in the sets of infected vertices, in this paper, a modification is examined in which infected vertices can return to a healthy state.  Vertices are initially infected independently at random and the central question is whether all vertices eventually become infected.  The model examined here is such a process on a square grid for which healthy vertices with at least two infected neighbours become infected and infected vertices with no infected neighbours become healthy.  Sharp thresholds are given for the critical probability of initial infections for all vertices eventually to become infected.
\end{abstract}

\section{Introduction}\label{S:intro}

Bootstrap percolation is a type of cellular automata on graphs in which vertices, often called `sites', can be in one of two possible states: `infected' or `uninfected' and the states are updated according to a local rule depending on their neighbourhoods.  Starting from a random configuration of infected and healthy vertices, the states of vertices are updated repeatedly in discrete time steps and one would like to know, in terms of the density of the initial infection, whether is it likely or unlikely that every vertex is eventually infected.  A common type of bootstrap process is the $r$-neighbour bootstrap process, where infected vertices remain infected forever and healthy vertices with at least $r$ infected neighbours become infected.  Bootstrap percolation was introduced by Chalupa, Leath, and Reich \cite{CLR79} as a monotone model for certain physical processes.  In this paper, a non-monotone modification of the $r$-neighbour bootstrap process is considered.  In particular, a sharp threshold for total infection is given for the process on a square grid where healthy vertices with at least $2$ infected neighbours become themselves infected and infected vertices with no infected neighbours become healthy, though they are susceptible to being re-infected at a later time.

For either the $r$-neighbour bootstrap rule or its modification described here, a set of initially infected vertices is said to \emph{percolate} if, in the sequence of state updates, every vertex in the graph is eventually infected and stays infected.  As configurations of infected states often seem difficult to describe, a typical case is considered in which vertices are infected at random, independently with some fixed probability.  The question of interest is, for which values for the initial infection probability percolation is more likely than not.  For the $r$-neighbour bootstrap process on a graph $G$, given initially infected sets, $X_0$, chosen independently at random with probability $p$, the \emph{critical probability} is defined as
\[
p_c(G, r) = \inf\{p \mid \mathbb{P}_p(X_0 \text{ percolates in $r$-nbr boots.}) \geq 1/2\}.
\]

The behaviour of the $r$-neighbour bootstrap processes has been well-studied for integer lattices in various dimensions, as well as finite subgraphs of these.  For the infinite integer lattices, $\mathbb{Z}^d$, it was shown by van Enter \cite{avE87} and Schonmann \cite{rS92} that the critical probabilities are all either $0$ or $1$.  For finite grids in $2$ dimensions, $[1,n]\times [1,n]$, Aizenman and Lebowitz \cite{AL88} first noted metastability effects of $2$-neighbour bootstrap percolation and gave bounds on the critical probability. A sharp bound for the critical probability was given by Holroyd \cite{aH03} who proved that for $n \in \mathbb{N}$, $p_c([n]^2,2) = (\pi^2/18 +o(1))/\log n$.  Subsequently, an even sharper result was given by Gravner, Holroyd and Morris \cite{GHM10}.  For grids in higher dimensions, sharp bounds on the critical probability were given by Balogh, Bollob\'{a}s and Morris \cite{BBM09b} and by Balogh, Bollob\'{a}s, Duminil-Copin and Morris \cite{BBDCM12}.


A useful feature in the analysis of $r$-neighbour bootstrap percolation is that once a site becomes infected, it remains infected.  There are, however, a number of closely related models where the set of `infected' vertices is not monotone, with respect to containment, over time.  Schonmann \cite{rH90} gave bounds on the critical probability for the biased majority rule on finite integer lattices, where at each time step, vertices change their state to match the majority of their neighbours, with the tie-break biased towards infection. Balister, Bollob\'{a}s, Johnson and Walters \cite{BBJW10} considered a random majority percolation which combined a majority  update rule together with a random perturbation, using techniques of the monotone bootstrap processes.  Fontes, Schonmann and Sidoravicius \cite{FSS02} and later Morris \cite{rM09} used bootstrap percolation to give bounds on critical probabilities related to the low-temperature Glauber dynamics on $\mathbb{Z}^d$.

Here, an update rule is considered which has the property, as with the majority-type rules, that vertices with few infected neighbours will become uninfected.  While some vertices can repeatedly become infected and uninfected in this process, many of the techniques used to produce sharp thresholds for the probability of percolation on finite grids can be adapted to give sharp thresholds for this model also.

\begin{definition}
Let $n \in \mathbb{N}$ and $X \subseteq [n]^2$.  Define the \emph{recovery update rule}  as follows,
\begin{equation}\label{E:modrule}
\mathcal{R}(X) = \{\mathbf{x} \in X:\ |(N(\mathbf{x})\cup \{\mathbf{x}\}) \cap X| \geq 2\}.
\end{equation}
It is sometimes helpful to compare the effect of the process $\mathcal{R}$ to that of $2$-neighbour bootstrap.  Define, for the $2$-neighbour bootstrap update and any $X \subseteq [n]^2$,
\[
\mathcal{B}(X) = X \cup \{\mathbf{x} \mid |N(\mathbf{x}) \cap X| \geq 2\}.
\]
\end{definition}

From a configuration of initially infected sites, those healthy vertices with at least $2$ infected neighbours become infected under $\mathcal{R}$, but in contrast to usual bootstrap percolation, infected sites with no infected neighbours become un-infected, or `recover'.  This occurs simultaneously for all sites and the process is repeated.

Given $X_0 \subseteq [n]^2$, recursively define, for each $t \geq 0$, $X_{t+1} =\mathcal{R}(X_t)$.  As in the $2$-neighbour bootstrap process, the set $X_0$ is said to \emph{percolate with respect to $\mathcal{R}$} if there is $t_{0}$ such that $\mathcal{R}^{(t_0)}(X_0) =X_{t_0}= [n]^2$.  Unless otherwise specified, in this paper, a set will be said to percolate if it percolates in the process $\mathcal{R}$.  

For any set $A \subseteq [n]^2$, the notation $X \sim \operatorname{Bin}(A,p)$ will be used to denote a random configuration $X \subseteq A$ with each element of $A$ included in $X$ independently with probability $p$.  Define the critical probability function for $\mathcal{R}$ by
\[
p_c([n]^2, \mathcal{R}) = \inf\{p \mid \prob_p(X_0 \text{ percolates}) \geq 1/2\}.
\]

Unlike the usual $2$-neighbour bootstrap update rule, the sequence of configurations $X_0, \mathcal{R}(X_0), \mathcal{R}^{(2)}(X_0), \ldots$ is not, in general, monotone with respect to set containment. For example, if the grid is initially infected with a checkerboard pattern, the the sets $(X_t)_{t\geq 0}$ alternate between $X_0$ and $[n]^2 \setminus X_0$.  For this reason, it might not make sense to talk about the span of a set of infected sites in the recovery process, $\mathcal{R}$.  In the context of the recovery bootstrap process, the span of a set of vertices in $\mathcal{B}$ will occasionally be considered.

If an infected site $\mathbf{x} \in X_0$ has a neighbour in $X_0$, then both $\mathbf{x}$ and its neighbour remain infected in all subsequent sets $X_t$.  However, if $\mathbf{x}$ shares a corner with an infected site, then for every $t\geq 0$, $\mathbf{x} \in X_{2t}$, but it might be the case that $\mathbf{x} \notin X_{2t+1}$.  

%
%
%

In general, the configurations of infected sites that percolate with respect to the usual $2$-neighbour bootstrap process need not percolate with respect to the recovery update rule $\mathcal{R}$.  In fact, there are configurations of infected sites which percolate with respect to $\mathcal{B}$ but for which all vertices become uninfected in the recovery process $\mathcal{R}$.

In this paper, sharp bounds on the critical probability for the recovery update rule $\mathcal{R}$ are determined, together with estimates on the probability of percolation.

These results are proved by considering a closely related process where isolated vertices, which disappear after one time step in $\mathcal{R}$, do not occur to begin with.  Given a fixed $p$ and $n$, consider infecting each pair of sites in $[1,n]^2$ that either share an edge or a corner, as in Figure \ref{F:2tiles}, with probability $p^2$.  Such pairs of sites are called \emph{$2$-tiles} and the probability of overlapping $2$-tiles when sites are infected independently at random is of order $p^3$, which is much smaller than $p^2$ for the values of $p$ considered here.  These $2$-tiles are well-behaved in the $\mathcal{R}$ process since an infected vertex in a tile that becomes uninfected will become re-infected at the following time-step.  By adapting the methods used for analyzing critical probabilities, for example by Balogh and Pete \cite{BP98}, one could show that the critical probability for such a process is $\Theta((\log n)^{-1/2 +o(1)})$.  In this paper, we in fact prove sharp results about the behaviour of $\mathcal{R}$ on $2$-tiles.

Using an approach similar to that used for $r$-neighbour bootstrap processes (see, e.g. \cite{aH03,BBM09b,BBDCM12}), the probability of the growth of the process is described in terms of an implicitly defined function arising from a careful recursion on the probability that there is a pair of adjacent columns in a rectangle with no site infected.  For each $u \in (0,1)$, $\beta(u)$ is defined to be the largest real root of the function $F(u, x) = x^3 - (1-u^4)x^2-u^4(10u^4)x - u^8(1-u^3)$.  Define the constant
\[
\lambda_{\mathcal{R}} = \int_{0}^{\infty} -\log(\beta(e^{-x}))\ dx.
\]
The critical probability for percolation of a random configuration of infected $2$-tiles is given by the following theorem.

\begin{theorem}\label{T:crit-p-tiles}
Let $\varepsilon > 0$ and $\{p(n)\}_{n \in \mathbb{Z}^+} \subseteq (0,1)$.  For each $n$, let $X_{\text{tiles}}(n)$ be a random configuration of $2$-tiles in $[1,n]^2$ with each tile infected independently at random with probability $p(n)$.  In the $\mathcal{R}$-process,
\[
\prob(X_{\text{tiles}}(n) \text{ percolates}) = 
	\begin{cases}
		 o(1) &\text{if } p(n)^2 < \frac{\lambda-\varepsilon}{\log n}\\
		 1-o(1)	&\text{if } p(n)^2 > \frac{\lambda+\varepsilon}{\log n}.
	\end{cases}
\]
\end{theorem}

What remains is the task of showing that the process where single sites are infected independently at random can be approximated by the process on independent $2$-tiles.  A central challenge is the lack of independence between nearby sites after one step in the $\mathcal{R}$-process, after which any configuration consists of $2$-tiles.  When considering both upper and lower bounds on the probability of percolation in the $\mathcal{R}$-process, the initially infected set is altered to create a set on which the rule $\mathcal{R}$ is nearly monotone and for which the probability of percolation has not changed too much. This alteration is done in different ways for each case, to obtain the main result of this paper.

\begin{theorem}\label{T:mainthm}
There exists a constant $\lambda_{\mathcal{R}} >0$ such that for every $\varepsilon >0$, and sequence $\{p(n)\}_{n \in \mathbb{Z}^+} \subseteq (0,1)$, in the bootstrap process with recovery, $\mathcal{R}$,
\[
\mathbb{P}_{p(n)}(X_0\ \mathcal{R}\text{-percolates}) = 
\begin{cases}
	1-o(1)	&\text{if } p(n) > \sqrt{\frac{\lambda_{\mathcal{R}} + \varepsilon}{\log n}}\\
	o(1)		&\text{if } p(n) < \sqrt{\frac{\lambda_{\mathcal{R}} - \varepsilon}{\log n}}.
\end{cases}
\]
\end{theorem} 

The probabilities of percolation for each of the two ranges of values for $p(n)$ in Theorem \ref{T:mainthm} give the following immediate formulation for the critical probability for the recovery bootstrap update rule.

\begin{corollary}\label{cor:crit-prob}
For all $n \in \mathbb{N}$,
\[
p_c(n, \mathcal{R}) = \sqrt{\frac{\lambda_{\mathcal{R}} }{\log n}}+ \frac{o(1)}{\sqrt{\log n}}.
\] 
\end{corollary}

The remainder of the paper is organized as follows.  In Sections \ref{ss:notation} and \ref{ss:prob-tools}, some notation is given that is used throughout the paper and some probabilistic tools are stated that are used repeatedly.  In Section \ref{S:2tiles}, the model with infection of $2$-tiles is defined precisely and results on the growth-rate of infection by $2$-tiles are given.  These results imply Theorem \ref{T:crit-p-tiles}.  

In Section \ref{S:lb}, one model of the original process by $2$-tiles is used to give sharp lower bounds on the probability of percolation.  These lead to an upper bound on the critical probability $p_c(n, \mathcal{R})$.  In Section \ref{S:ub}, a different model of the process on sites by a process on $2$-tiles is used to obtain upper bounds on the process of infection spreading within a rectangle and a notion of `hierarchies' for percolation suited to the $2$-tile model is defined and used to account for the different ways percolate in the $\mathcal{R}$-process.  These lead to an upper bound on the probability of percolation and hence a lower bound on the critical probability.  Together, these results prove Theorem \ref{T:mainthm} and Corollary \ref{cor:crit-prob}.

\subsection{Notation}\label{ss:notation}

Because of the different types of pairs of sites, it is often useful to consider both pairs of sites sharing an edge and pairs of sites sharing a corner as neighbours of different types.  For $r\geq0$ and a site $\mathbf{x} \in [n]^2$, define two different balls of radius $r$ in the grid, centered at $\mathbf{x}$, 
\begin{align}
	B_r(\mathbf{x})	&=\{\mathbf{y} \in [n]^2: \|\mathbf{x}-\mathbf{y}\|_1 \leq r\}, \label{D:ell1ball}\\
	B_r^*(\mathbf{x})	&=\{\mathbf{y} \in [n]^2: \|\mathbf{x}-\mathbf{y}\|_{\infty} \leq r\} \label{D:ellmaxball}.
\end{align}
For any $\mathbf{x} \in [n]^2$, the set $B_1(\mathbf{x})$ is precisely the set $\{\mathbf{x}\}\cup N(\mathbf{x})$ while the set $B_1^*(\mathbf{x})$ is the set containing $\mathbf{x}$ together with the sites either sharing an edge or corner with $\mathbf{x}$.  Unless otherwise stated, all distances between sites or sets of sites are given by the $\ell_1$ norm.

Often, it is not just square grids that are of interest, but also `rectangles' contained in the grid.  A set $R \subseteq [1,n]^2$ is called a \emph{rectangle} if there are $a_1< a_2 \in [1,n]$ and $b_1<b_2 \in [1,n]$ with $R = [a_1, a_2]\times [b_1, b_2]$.  A rectangle $R$ is said to be \emph{internally spanned} by the initially infected sites $X_0$ if there is a $t_{R}$ so that $\mathcal{R}^{(t_R)}(X_0 \cap R) = R$.   In other words, $R$ is internally spanned if based only on the initially infected sites inside the rectangle $R$, every site in $R$ eventually becomes infected.  For $p \in (0,1)$, and $X_0 \sim \operatorname{Bin}(R,p)$, let $I(R,p)$ denote the probability that the rectangle $R$ is internally spanned.  To simplify notation, let $I(n,p)$ denote the probability that $[n]^2$ is internally spanned.

A rectangle $R = [a_1, a_2]\times[b_1, b_2]$ is said to be \emph{horizontally traversable from left to right} by $X_0$ if $R \setminus (\{a_2\} \times [b_1, b_2])\cup (\{a_1-1\}\times[b_1, b_2])$ is internally spanned by $X_0 \cup \{a_1-1\}\times[b_1, b_2]$.  That is, assuming all sites in the column $\{a_1-1\} \times [b_1, b_2]$ are infected then the sites in $X_0$ will cause the infection to spread to all of $R$ except possibly the final column, depending only on the sites that are infected inside the rectangle $R$.  The events that the rectangle $R$ is horizontally traversable from right to left,  vertically traversable from bottom to top, or vertically traversable from top to bottom are defined similarly.

The following notation for rectangles is used throughout. For a rectangle $R = [a_1, a_2]\times [b_1, b_2]$, the \emph{dimensions of $R$}, denoted by $\dim(R)$ is the pair of side-lengths of $R$: $\dim(R) = (a_2 - a_1+1, b_2-b_1+1) $.  The length of the shorter side of $R$ is denoted $\text{short}(R) = \min\{a_2 - a_1+1, b_2-b_1+1\}$, the length of the longer side of $R$ is denoted $\text{long}(R) = \max\{a_2 - a_1+1, b_2-b_1+1\}$ and the semi-perimeter of $R$ is $\phi(R) = (a_2-a_1 +1)+(b_2-b_1+1) = \text{short}(R) + \text{long}(R)$. 

\subsection{Probabilistic tools}\label{ss:prob-tools}

A few probabilistic results are used repeatedly throughout and those are stated here for reference.  The following Chernoff-type bound gives estimates for the likelihood of a binomial random variable being much larger than its mean.  In the following form it can be found, for example, in \cite[pp 11--12]{bB01}.  If $n\in \mathbb{N}^+$ and $p \in (0,1)$, then for $pn < m <n$,
\begin{equation}\label{E:binomtail}
\prob(\operatorname{Bin}(n,p) \geq m) 	\leq e^{m-np} \left(\frac{np}{m}\right)^m.
\end{equation}

Two results on increasing and decreasing events in the cube $Q^n_p$ can be found, for example, in Bollob\'{a}s and Riordan \cite[pp. 39--44]{BR06}.  Harris's Lemma \cite{tH60} states that if $A, B \subseteq Q_p^n$ are both increasing events or if both $A$ and $B$ are decreasing events, then
\begin{equation}\label{E:FKG_Harris}
\prob(A \cap B) \geq \prob(A)\ \prob(B).
\end{equation}
For sets $A,B \subseteq Q^n$, let $A \Box B \subseteq A \cap B$ be the event that $A$ and $B$ occur disjointly. The van den Berg-Kesten inequality \cite{vdBK85} states that if $A, B \subseteq Q_n^p$ are either both increasing events or both decreasing events, then
\begin{equation}\label{E:vdBK-ineq}
\prob(A \Box B) \leq \prob(A) \prob(B).
\end{equation}
Reimer \cite{dR00} showed that, in fact, inequality \eqref{E:vdBK-ineq} holds for any $A, B \subseteq Q_n^p$.  However, in this paper, we shall only require the inequality \eqref{E:vdBK-ineq} for increasing or decreasing events.

The following inequality, due to McDiarmid and Reed \cite{McR06}, is a variation on a concentration result by Talagrand \cite{mT95}.  Further details on the results are given, for example, by Talagrand in \cite{mT96}.  Fix $c >0$, $r\geq 0$, $d\geq 0$ and let $X_1, \ldots, X_n$ be independent Bernoulli random variables.  Let $g=g(X_1, \ldots, X_n)$ be a function with mean $\mu$ such that
\begin{enumerate}[(a)]
	\item if $\mathbf{x}, \mathbf{x}' \in \{0,1\}^n$ differ in exactly one coordinate, then $|g(\mathbf{x})-g(\mathbf{x}')| \leq c$ and
	\item for any $s \geq 0$, if $g(\mathbf{y}) \geq s$, there is a set $I \subseteq [1,n]$ with $|I| \leq rs +d$ such that if $\mathbf{y}' \in \{0,1\}^n$ agrees with $\mathbf{y}$ on the coordinates in $I$, then $g(\mathbf{y}') \geq s$.
\end{enumerate}
Then, for any $t \geq 0$
\begin{equation}\label{T:tal}
\prob(g - \mu \geq t) 	 \leq \exp\left({-\frac{t^2}{2c^2(r\mu + d + rt)}}\right).
\end{equation}

The following inequality due to Janson \cite{sJ90}, found also, for example, in \cite[p.33]{JLR00} can be used to estimate the probability of combinations of certain events with a limited amount of dependency.  Fix $n \in \mathbb{Z}^+$, let $\mathbf{p} \in (0,1)^n$, and let $\{A_i:\ i \in I\}$ be a finite collection of subsets of $\{1, 2, \ldots, n\}$.  Choose $\mathbf{x} \in Q_{\mathbf{p}}^n$ at random, according to the product measure given by $\mathbf{p}$ and for each $i \in I$, let $B_i \subseteq Q_{\mathbf{p}}^n$ be the event that for every $j \in A_i$, $x_j=1$.  Set
\[
\Delta = \sum_{\underset{A_i \cap A_j \neq \emptyset}{B_i, B_j}} \prob_{\mathbf{p}}(B_i \cap B_j)
\]
and let $\mu = \sum_{i \in I}\prob_{\mathbf{p}}(B_i)$.  Then
\begin{equation}\label{E:Jansonineq}
\prob_{\mathbf{p}}(\cap_{i \in I} \overline{B}_i) \leq e^{-\mu + \Delta}.
\end{equation}

\section{Infection with pairs of sites}\label{S:2tiles}

As described in Section \ref{S:intro}, the effect of the process $\mathcal{R}$ on infected sites with an infected neighbour can be more easily understood than the effect of $\mathcal{R}$ on sites with no infected neighbours.  There is still considerable difficulty in dealing with sites with infected neighbours since the events that two particular sites both have infected neighbours are not, in general, independent events.  With this in mind, a new infection scheme is defined where pairs of neighbouring sites are infected simultaneously.

For each $\mathbf{x} \in [n]^2$, consider the four pairs of sites
\begin{align}
&T_{(1,1)}(\mathbf{x}) = \{\mathbf{x}, \mathbf{x} + (1,1)\},			&T_{(1,-1)}(\mathbf{x}) = \{\mathbf{x}, \mathbf{x} + (1,-1)\}, \notag\\
&T_{(1,0)}(\mathbf{x}) = \{\mathbf{x}, \mathbf{x} + (1,0)\}, \text{ and }	&T_{(0,1)}(\mathbf{x}) =\{\mathbf{x}, \mathbf{x} + (0,1)\}. \label{E:2tiles}
\end{align}
Call each of these pairs of sites a \emph{$2$-tile}.  In order to be precise about the position of such pairs, for each $2$-tile in \eqref{E:2tiles}, call $\mathbf{x}$ the \emph{anchor} of the $2$-tile.  The anchor of a $2$-tile is the left-most, bottom-most site.  In Figure \ref{F:2tiles}, these are the black squares while the non-anchor sites are grey squares. 

The $2$-tiles of the first three types, $T_{(1,1)}(\mathbf{x})$, $T_{(1,0)}(\mathbf{x})$, and $T_{(1,-1)}(\mathbf{x})$ are said to be of \emph{width 2} while the $2$-tiles of the last type, $T_{(0,1)}(\mathbf{x})$ are said to be of \emph{width 1}.

\begin{figure}[h]
	\begin{center}
\begin{tikzpicture}[line width=1.5pt, scale = 0.8]
	\draw[step=0.8cm] (-1.2,-1.2) grid (1.2,1.2);
	\filldraw[fill=black] (-0.8,-0.8) rectangle (0,0);
	\filldraw[fill=gray, draw=black] (0,0) rectangle (0.8,0.8);
	
	\draw[step=0.8cm] (2,-1.2) grid (4.4,1.2);
	\filldraw[fill=black] (2.4,0.8) rectangle (3.2,0);
	\filldraw[fill=gray, draw=black] (3.2,0) rectangle (4,-0.8);
	
	\draw[step=0.8cm] (5.2,-1.2) grid (7.6,0.4);
	\filldraw[fill=black] (5.6,-0.8) rectangle (6.4,0);
	\filldraw[fill=gray, draw=black] (6.4,0) rectangle (7.2,-0.8);
	
	\draw[step=0.8cm] (8.4,-1.2) grid (10,1.2);
	\filldraw[fill=black] (8.8,-0.8) rectangle (9.6,0);
	\filldraw[fill=gray, draw=black] (8.8,0) rectangle (9.6,0.8);
	
\end{tikzpicture}

	\end{center}
		\caption{Pairs of sites forming $2$-tiles}
		\label{F:2tiles}
\end{figure}
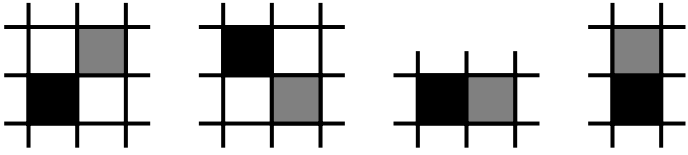

Given $p>0$, let $X_{\text{tiles}}$ be a random configuration of $2$-tiles on the grid $[n]^2$ with each of the $2$-tiles with anchor in $[n]^2$:
\[
\bigcup_{\mathbf{x} \in [n]^2}\{T_{(1,1)}(\mathbf{x}),T_{(1,-1)}(\mathbf{x}), T_{(1,0)}(\mathbf{x}), T_{(0,1)}(\mathbf{x}) \}
\]
included independently with probability $p^2$.  Note that, in general, $X_{\text{tiles}}$ might contain many overlapping $2$-tiles.  To avoid confusion, the measure on configurations of $2$-tiles on the grid is denoted $\mathbb{P}_2$.  

Any configuration of $2$-tiles is naturally associated with the set of sites in the grid that are contained in some $2$-tile.  A configuration of $2$-tiles, $X_{\text{tiles}}$ is said to \emph{percolate} if the set of sites in some $2$-tile of $X_{\text{tiles}}$ percolates.  Similarly a rectangle $R$ will be said to be traversable in some particular direction with respect to $X_{\text{tiles}}$ if $R$ is traversable in that direction by the set of sites in some $2$-tile.  

It is shown, in Sections \ref{S:lb} and \ref{S:ub}, that the probability that a random configuration of infected sites percolates (in $\mathcal{R}$) can be approximated by the probability that a random configuration of $2$-tiles percolates.

One advantage to working with $2$-tiles is that, since every infected site has a neighbour, either along an edge or at a corner, a rectangle $R$ is traversable by $X_{\text{tiles}}$ under $\mathcal{R}$ exactly when $R$ is traversable by $X_{\text{tiles}}$ with respect to the $2$-neighbour bootstrap process.  Thus, as in usual $2$-neighbour bootstrap percolation, the only obstacle to crossing $R$ is a pair of adjacent columns containing no infected sites.

Given a configuration of $2$-tiles, a column is called \emph{$2$-occupied} if it contains the anchor of a $2$-tile of width 2, a column is \emph{$1$-occupied} if it contains the anchor of a $2$-tile of width 1, and a column is \emph{unoccupied} if it does not contain the anchor of any $2$-tile.  A column is said to be \emph{occupied} if it is either $1$-occupied or $2$-occupied.  Note that a column might be unoccupied and yet contain the non-anchor of some $2$-tile.  Call a column \emph{empty} if it does not contain any sites from any $2$-tiles. A pair of empty adjacent columns is called a \emph{double gap}.

As in the study of usual bootstrap percolation (see for example, Holroyd \cite[Lemma 7]{aH03}), the probability that a rectangle $R$ contains no double gaps is defined recursively in terms of the number of columns in $R$.  The following function appears as the characteristic function of the recurrence relation that arises in the analysis of infection by $2$-tiles and a few helpful facts are first proved.

\begin{definition}
For each $u \in (0,1)$, set
\begin{align*}
F(u, x) = F_{u}(x) 	&=  x^3-(1-u^4)x^2-u^4(1-u^4)x - u^8(1-u^3)\\
					&= (x-1)(x^2+u^4x + u^8) + u^{11}
\end{align*}
and let $\beta(u)$ be the largest real root of $F_u(x)$.
\end{definition}

Since $F_u(0) = -u^8(1-u^3) < 0$ and $F_u(1) = u^{11} > 0$, there is at least one root in $(0,1)$. Consider the derivative 
\[
\frac{\partial}{\partial x}F(u, x) = F'_u(x)=3x^2-2(1-u^4)x-u^4(1-u^4).
\]
 As $F'_u(0)=-u^4(1-u^4)<0$ and $F'_u(1)=1+u^4+u^8>0$, the function $F_u$ has a relative maximum less than zero at a value less than zero and a relative minimum between 0 and 1. Since $F_u(x)$ is a polynomial  of degree $3$ in $x$, $F_u(x)$ has exactly one root in $(0,1)$. In order to obtain bounds on the value of $\beta(u)$ in terms of $u$, note that  
\[
F_u(1-u^{11}) 	= -u^{15}(1+u^4-2u^7-u^{11} + u^{18})<0,
\]
and hence $1-u^{11} \leq \beta(u) \leq  1$.  With a little more work, it can be shown that for any $u \in (0,1)$, $\beta(u) \in (1-u^{11}, (6(1-u))^{1/3})$. 


\begin{lemma}\label{L:crossing}
Fix $p \in (0,1)$ and let $R$ be a rectangle of dimension $(m,h)$.  Set $u = (1-p^2)^h$ and let $X_{\text{tiles}}$ be a random configuration of $2$-tiles, each included independently with probability $p^2$.  Then, 
\[
(1-u^8) \beta(u)^{m} \leq \mathbb{P}_2(R \text{ is horizontally traversable by } X_{\text{tiles}}) \leq \beta(u)^{m-1}.
\]
\end{lemma}

\begin{proof}
Fix $h \geq 1$ and set $u = u(p,h)=(1-p^2)^h$.  Let $C$ be any column of sites of height $h$, a rectangle of dimension $(1,h)$, then 
\begin{align*}
\mathbb{P}_2(C \text{ is $1$-occupied}) 	&=1-(1-p^2)^h = 1-u,\\
\mathbb{P}_2(C \text{ is $2$-occupied})	&=1-(1-p^2)^{3h} = 1-u^3,\\
\mathbb{P}_2(C \text{ is either $1$ or $2$-occupied})	&= 1-(1-p^2)^{4h} = 1-u^4, \text{ and}\\
\mathbb{P}_2(C \text{ is unoccupied})		&=(1-p^2)^{4h} = u^4.
\end{align*}

Considering only the squares inside the relevant rectangle, for each $m \geq 0$, let $R_m=[m]\times [h]$ and set 
\[
A_m = \mathbb{P}_2(R_m \text{ horiz. trav.}) = \mathbb{P}_2(X_{\text{tiles}} \text{ has no double gaps in } R_m).
\]
  In order to obtain bounds on the value of $A_m$, a recursion for the sequence $\{A_m\}_{m \geq0}$ is defined.  Let the columns of $R$ be denoted $C_1, C_2, \ldots, C_m$.  When $m \geq 3$, there are three distinct ways to traverse a rectangle of width $m$:
\begin{enumerate}[(a)]
   \item either $C_m$ is occupied and $R\setminus C_m$ is traversable, 
   \item $C_m$ is unoccupied, $C_{m-1}$ is occupied and $R \setminus (C_{m-1} \cup C_m)$ is traversable, or finally
   \item $C_{m-1}$ and $C_m$  are both unoccupied, the column $C_{m-2}$ is $2$-occupied and the first $m-3$ columns of $R$ are traversable.
\end{enumerate}
     The first few values of $A_m$ can be calculated exactly:
$A_0	=1$,
$A_1	=1$, and 
$A_2	 =1- \mathbb{P}_2(C_1\text{ is unoccupied})^2 = 1-u^8$.

Considering the three cases above, for each $m \geq 3$, a recurrence relation for the sequence $\{A_m\}_{m \geq 0}$ is given by
\[
A_m = (1-u^4)A_{m-1} + u^4(1-u^4)A_{m-2} + u^8(1-u^3)A_{m-3}.
\]

Recall that $\beta(u)$ is a real root in $(0,1)$ of the polynomial $F_u(x) = x^3-(1-u^4)x^2 - u^4(1-u^4)x - u^8(1-u^3)$.  Instead of solving the recursion exactly, 
the goal is to show, by induction, that for all $m$, the value of $A_m$ is close to $\beta(u)^m$.

The base cases can be checked directly,
\begin{align*}
(1-u^8)\beta(u)^0 = (1-u^8) &< 1 = A_0 	< \beta(u)^{-1} \text{ and}		&&\text{(since  $\beta(u)^{-1}>1$)}\\
(1-u^8) \beta(u) &< 1 = A_1		= \beta(u)^0.
\end{align*}
Since $\beta(u) < 1$, then $(1-u^8)\beta(u)^2< (1-u^8) = A_2$ and since $u \in (0,1)$, $A_2 = 1-u^8 < 1-u^{11} < \beta(u)$. The rest follows by induction, using the fact that $\beta(u)$ is a zero of the characteristic equation for the recurrence for the sequence $\{A_m\}$.
\end{proof}

In what follows, a few basic properties of the function $\beta(u)$ are used: the rough bounds already given and the properties stated in the following lemma.

\begin{lemma}
In the interval $(0,1)$, the function $\beta(u)$ is continuous, decreasing, and concave.
\end{lemma}

\begin{proof}
By the implicit function theorem, since $F(u,x)$ is a continuously differentiable function, so is $\beta(u)$ on any open interval for which $\frac{\partial F}{\partial x} (u, \beta(u)) \neq 0$.  Now
\begin{align*}
\frac{\partial F}{\partial x} (u,x)	
	&=3x^2 - 2(1-u^4)x - u^4(1-u^4)\\
	&=\frac{3}{x} \left(F(u,x) + \frac{1}{3}(1-u^4)x^2 + \frac{2}{3}u^4(1-u^4)x + u^8(1-u^3)\right)\\
	&> \frac{3}{x}F(u,x).		&& \hspace{-70pt}\text{(for $u, x \in (0,1)$)}
\end{align*}
Thus, for any $u \in (0,1)$, $\frac{\partial F}{\partial x}(u, \beta(u)) > \frac{3}{\beta(u)} F(u, \beta(u)) = 0$.  Further,
\begin{align*}
\frac{\partial F}{\partial u}(u,x)
	&=4u^3x^2 - 4u^3 x + 8u^7x - 8u^7 + 11u^{10}\\
	&=4u^3(x^2 - (1-2u^4)x - 2u^4(1-11/8 u^3)).
\end{align*}
Since for all $u \in (0,1)$, $\beta(u) > 1-u^{11} > 1-2u^4$,
\begin{align*}
\frac{\partial F}{\partial u}(u, \beta(u))
	&=4u^3 \big( \beta(u)(\beta(u) - (1-2u^4)) - 2u^4(1-11/8 u^3) \big)\\
	&> 4u^3 \big( (1-u^{11})(1-u^{11} - 1 + 2u^4) - 2u^4 + 11/4u^7 \big)\\
	&=4u^{10}(3/4 + (1-u^7)(1-u^8) + (1-u^4) + u^7(1-u))\\
	&>0.
\end{align*}
Thus, for all $u \in (0,1)$, the function $\beta(u)$ is differentiable and hence continuous with 
\[
\beta'(u) = -\frac{\frac{\partial F}{\partial u}(u, \beta(u))}{\frac{\partial F}{\partial x}(u, \beta(u))} <0
\]
and hence $\beta(u)$ is decreasing.

To see that $\beta$ is concave note that by differentiating implicitly,
\[
\beta''(u) = -\left(\frac{\frac{\partial^2 F}{\partial u^2}\left(\frac{\partial F}{\partial x}\right)^2 + \frac{\partial^2 F}{\partial x^2}\left(\frac{\partial F}{\partial u}\right)^2 - 2\frac{\partial^2 F}{\partial u \partial x}\frac{\partial F}{\partial x}\frac{\partial F}{\partial u}}{\left(\frac{\partial F}{\partial x}\right)^3}\right)(u, \beta(u)).
\]
As above, $\left(\frac{\partial F}{\partial x}\right)^3(u, \beta(u)) > 0$ and  a straight-forward, but tedious calculation shows that the expression in the numerator is also positive.

\end{proof}

Following notation similar to that used in the study of usual bootstrap percolation, it is often convenient to use the following functions related to $\beta(u)$.  Set 
\[
g(x) = -\log{(\beta(e^{-x}))}.
\]
  For $p \in (0,1)$, define 
\[ 
  q = q(p) = (-\log{(1-p^2)})^{1/2}.
\]
    When $p$ is small, $q^2 \sim p^2$, with the advantage that for any $p>0$ and  $h \in \mathbb{Z}^+$,
\begin{equation}\label{E:betag}
\beta((1-p^2)^h) = e^{-g(hq^2)}.
\end{equation}

Since $\beta(u)$ is defined for $u \in [0,1]$, $g(x)$ has domain $x \in (0, \infty)$ and has the following useful properties.

\begin{fact}\label{Fact:gconvex}
The function $g(x) = -\log{(\beta(e^{-x}))}$ is decreasing, convex and integrable on $(0,\infty)$.
\end{fact}

\begin{proof}
 Since the function $\beta(u)$ is decreasing in $u$, then its partial derivative with respect to $x$ is $\frac{\partial}{\partial x} (\beta(e^{-x})) = -e^{-x}\beta(e^{-x}) >0$.  Thus, since $-\log{x}$ is decreasing in $x$, the function $g$ is decreasing.  

The function $\beta(e^{-x})$ is concave since
\[
\frac{\partial^2}{\partial x^2}(\beta(e^{-x})) = e^{-x}(\beta''(e^{-x})e^{-x} + \beta'(e^{-x}))<0
\]
and hence, since $-\log x$ is both decreasing and convex, $g(x) = -\log(\beta(e^{-x}))$ is convex as well.
To see that $g$ is integrable, note that since $\beta(e^{-x}) \geq 1-e^{-11x}$, then $g(x) \leq -\log{(1-e^{-11x})}$ and so
\[
	\int_{0}^{\infty} g(x)\ dx
		\leq \int_0^{\infty} -\log(1-e^{-11x})\ dx\\
		 = \sum_{k \geq 1} \frac{1}{11k^2} = \frac{\pi^2}{66} < \infty.
\]
Thus, $g$ is convex and $\int_0^{\infty} g(x)\ dx < \infty$.
 \end{proof}

A sequence of constants obtained by integrals of the function $g$ are key elements of the proofs to come and are denoted as follows.

\begin{definition}
Denote the integral of $g$ over $(0, \infty)$ by
 \[
 \lambda = \lambda_{\mathcal{R}} = \int_{0}^{\infty} g(x)\ dx
 \]
  and for each $n >0$, set 
  \[
  \lambda_n = \int_{1/n}^n g(x)\ dx.
  \] 
\end{definition}

The exact value of $\lambda$ is not used in any of the proofs that follow, but it can be shown that 
$\lambda \approx 0.0779$.

The results of Holroyd \cite{aH03} on the critical probability for usual bootstrap percolation can be directly applied to the model of infection by $2$-tiles with the function $g$ as given in equation \eqref{E:betag} and $\lambda$ as above.  For $\{p(n)\}_{n \geq 1} \subseteq (0,1)$, let $X_{\text{tiles}}(n)$ a random configuration of $2$-tiles in $[n]^2$, with each $2$-tile included in $X_{\text{tiles}}$ independently with probability $p(n)^2$.  Then, for all $\varepsilon>0$, 
\[
\prob_2(X_{\text{tiles}}(n) \text{ percolates}) = 
	\begin{cases}
		 o(1) &\text{if } p(n)^2 < \frac{\lambda-\varepsilon}{\log n}\\
		 1-o(1)	&\text{if } p(n)^2 > \frac{\lambda+\varepsilon}{\log n}.
	\end{cases}
\]

It remains to show that, indeed, the model of infection by $2$-tiles is a good approximation for the probability of percolation in the recovery process $\mathcal{R}$ when single sites are initially infected.  In Sections \ref{S:lb} and \ref{S:ub}, two different alterations of an initially infected set of sites are given to obtain lower and upper bounds, respectively, on the probability of percolation.

\section{Lower bound for probability of percolation}\label{S:lb}

\subsection{Traversing rectangles}

In this section, it is shown that for certain values of $p$ and $n$, it is very likely that the grid, $[n]^2$, percolates in the bootstrap process with recovery when sites are initially infected independently with probability $p$.

 Given a configuration of infected sites $X$, a new configuration $X^{-}$ is defined so that $X^{-} \subseteq X$ and with the property that every site in $X^{-}$ has a neighbour in $X^{-}$ either sharing an edge or a corner.  The configuration $X^{-}$ can then be compared to configurations of $2$-tiles.  This is accomplished most simply in the cases where there is no ambiguity with regards to assigning $2$-tiles to pairs of sites in $X^{-}$.

Throughout, let $X_0$ be the set of initially infected sites; each site infected independently with probability $p$.  As before, let $X_{\text{tiles}}$ be a configuration of $2$-tiles on the sites in $R$ with each $2$-tile occurring independently with probability $p^2$.  Given a configuration of $2$-tiles, $X_{\text{tiles}}$, define $|X_{\text{tiles}}|$ to be the number of squares in the grid that are contained in at least one $2$-tile.  If there is a site is contained in more than one $2$-tile of the configuration $X_{\text{tiles}}$, the site is only counted once for $|X_{\text{tiles}}|$.  

A configuration of sites, $X_0$, where every site is contained in some $2$-tile can be most easily compared to a configuration of $2$-tiles if $X_0$ determines exactly one configuration of $2$-tiles.  With this in mind, it will be useful to keep track of pairs of $2$-tiles that could cause ambiguity.   
Recall that $B_1^*$ is used for the $\ell_\infty$ ball of radius $1$.

\begin{definition}
A pair of $2$-tiles $\{\mathbf{x}, \mathbf{x_1}\}$ and $\{\mathbf{y}, \mathbf{y_1}\}$ forms a \emph{triple} if 
\[
B_1^*(\{\mathbf{x}, \mathbf{x_1}\}) \cap \{\mathbf{y}, \mathbf{y_1}\} \neq \emptyset.
\]
\end{definition}
Thus, two tiles that overlap form a triple and also two $2$-tiles that touch, either along an edge or at a corner, form a triple.  Each triple involves at least $3$ sites and so occurs in the set $X_0$ with probability at most $p^3$.

\begin{definition}\label{D:X*}
For any $n \in \mathbb{N}$ and $X \subseteq [n]^2$, define $X^{-} \subseteq X$ as follows:
\[
X^- = \{\mathbf{x} \in X:\ B_1^*(\mathbf{x}) \cap X \neq \{\mathbf{x}\}\}.
\]
If $\mathbf{x} \in X$ and $B_1^*(\mathbf{x}) \cap X = \{\mathbf{x}\}$, call $\mathbf{x}$ an \emph{isolated site (of $X$)}.
\end{definition}

Since $X^{-} \subseteq X$, if $X^{-}$ percolates, then so does $X$.  However, since every site in $X^{-}$ has a neighbour in $X^{-}$ sharing an edge or a corner, the set $X^{-}$ can be compared to a configuration of $2$-tiles and the estimates from Lemma \ref{L:crossing} on the probability of traversing a rectangle can be used.  In the following lemma, a lower bound is given for the probability that a rectangle of a particular scale is horizontally traversable.  In further proofs, this lower bound is used for rectangles of height either slightly smaller or slightly larger than $p^{-2}$ and so rectangles are considered whose height is in the interval $[p^{-15/8}, p^{-17/8}]$.  In order to better control the approximations, only rectangles of width at most $p^{-1/4}$ are considered.


\begin{lemma}\label{L:crossing_P1}
There is a $p_0>0$ so that for all $p<p_0$, $h=h(p)$ with $p^{-15/8} \leq h \leq p^{-17/8}$, $m=m(p)$ with $1 \leq m \leq p^{-1/4}$ and rectangle $R$ of dimension $(m,h)$,  if $X_0 \sim \operatorname{Bin}(R, p)$ then,
\[
\prob(R \text{ is horiz. trav. by } X_0) \geq  e^{-463hmp^{5/2}}(1-(1-p^2)^{8h})e^{-mg(q^2h)}.
\]
\end{lemma}

\begin{proof}
Fix $p$, $h$, and $m$ with $p^{-15/8} \leq h \leq p^{-17/8}$ and $1 \leq m \leq p^{-1/4}$.  Let $R$ be a rectangle of dimension $(m,h)$  and define a set of configurations of $2$-tiles 
\begin{align*}
\mathcal{Q} &= \{A\ \mid A \text{ is a configuration of $2$-tiles in $R$, containing no triples with } \\
	&\hspace*{40pt} |A| \leq hmp^{3/2}\}.
\end{align*}
The configurations of $2$-tiles in $\mathcal{Q}$ are, essentially, those that can be unambiguously compared to configurations of infected sites.  In later estimates, it is useful to assume that $|X_{\text{tiles}}|$ is not too large and so the condition $|A| \leq hmp^{3/2}$ is included also.

Given a configuration $A$ of $2$-tiles, let $A_1$ be the set of sites that are contained in some $2$-tile from $A$ and let $\mathcal{Q}_1 = \{A_1\ \mid A \in \mathcal{Q}\}$ be the configurations of infected sites corresponding to the configurations of $2$-tiles in $\mathcal{Q}$.

First, it is shown that the probability of the event $|X_{\text{tiles}}|>hmp^{3/2}$  is relatively small.  If at least $hmp^{3/2}$ sites are covered by $2$-tiles in $X_{\text{tiles}}$, then at least $\frac{1}{2} hmp^{3/2}$ different $2$-tiles were included in $X_{\text{tiles}}$.  Note that, for $p$ sufficiently small, $4hmp^2 <  \frac{hm p^{3/2}}{2}$.  Thus, by tail estimates for binomial random variables given in  inequality \eqref{E:binomtail},
\begin{align*}
\prob_2(|X_{\text{tiles}}| > hm p^{3/2})	
	& \leq \prob\left(\operatorname{Bin}(4hm, p^2) > \frac{1}{2} hmp^{3/2}\right)\\
	& \leq \exp \biggl(\frac{hmp^{3/2}}{2}\biggr)(8p^{1/2})^{hmp^{3/2}/2}\\
	& \leq \exp \biggl(\frac{hmp^{3/2}}{2}\biggr)(e^{-5})^{hmp^{3/2}/2}
		&&\text{(for $p \leq e^{-10}/64$)}\\
	&=e^{-2hmp^{3/2}}.
\end{align*}
In order to compare this term with those involving $\beta(u)$, note that since $u = (1-p^2)^h \leq e^{-11p^2h}$ and $\beta(u) \geq 1-u^{11}$,
\[
\beta(u)		\geq 1-e^{-11p^2h}
			\geq 1-e^{-11p^{1/8}}
			\geq e^{-p^{-3/8}}
			\geq e^{-h p^{3/2}}.
\]
Thus, $\prob_2(|X_{\text{tiles}}| > hm p^{3/2}) \leq e^{-hmp^{3/2}} \beta(u)^m$.

Fix $A \in \mathcal{Q}$.  Since $A$ contains no triples, the configuration $A$ consists of exactly $|A|/2$ tiles.  Thus,
\begin{align*}
\prob_2(X_{\text{tiles}} = A)	&=(p^2)^{|A|/2} (1-p^2)^{4|R| - |A|/2}\\
				&=p^{|A|} (1-p^2)^{4|R| - |A|/2}.
\end{align*}

In order to bound the probability that $X_0^- = A_1$, note that $X_0^- = A_1$ if the following three events occur:
\begin{itemize}
	\item $E_1$: the event that $A_1 \subseteq X_0$,
	\item $E_2$: the event $(B_1^*(A_1) \setminus A_1) \cap X_0 = \emptyset$, and
	\item $E_3$: the event that every site $\mathbf{x} \in X_0 \setminus A_1$ is isolated.
\end{itemize}
Since $E_1$ is independent of $E_2 \cap E_3$, 
\[
\prob(X_1^- = A_1) = \prob(E_1 \cap E_2 \cap E_3) = \prob(E_1) \prob(E_2 \cap E_3) = p^{|A|} \prob(E_2 \cap E_3).
\]
To obtain an upper bound on $|B_1^*(A_1) \setminus A_1|$, note that if $\mathbf{x}_1$ and $\mathbf{x}_2$ are two sites sharing an edge, then $|B_1^*(\{\mathbf{x}_1, \mathbf{x}_2\}) \setminus \{\mathbf{x}_1, \mathbf{x}_2\}| = 10$ whereas if $\mathbf{x}_1$ and $\mathbf{x}_2$ are two sites sharing a corner, then $|B_1^*(\{\mathbf{x}_1, \mathbf{x}_2\}) \setminus \{\mathbf{x}_1, \mathbf{x}_2\}| = 12$.  In both cases, $|B_1^*(\{\mathbf{x}_1, \mathbf{x}_2\}) \setminus \{\mathbf{x}_1, \mathbf{x}_2\}| \leq 6|\{\mathbf{x}_1, \mathbf{x}_2\}|$.  In general, there are at most $6|A|$ sites  in $B_1^*(A_1) \setminus A_1$, and so
\[
\prob(E_2) = (1-p)^{|B_1^*(A_1) \setminus A_1|} \geq (1-p)^{6|A|}.
\]
The event $E_3$ is the intersection of a collection of decreasing events: that for each site outside of $A_1$, none of the $4$ possible sets of sites forming tiles is included in $X_1$.  Thus, by Harris's Lemma (inequality \eqref{E:FKG_Harris}), $\prob(E_3) \geq (1-p^2)^{4(|R| - |A|)}$.  Since $E_2$ and $E_3$ are both decreasing events, applying inequality \eqref{E:FKG_Harris} again yields
\[
\prob(E_2 \cap E_3) \geq \prob(E_2)\prob(E_3) \geq (1-p)^{6|A|}(1-p^2)^{4(|R| - |A|)}.
\]
Thus,
\begin{align}
\prob(X_1^- = A_1)	
	&\geq p^{|A|}(1-p)^{6|A|}(1-p^2)^{4(|R|-|A|)} \notag\\
	& = p^{|A|}(1-p^2)^{4|R| - |A|/2}(1-p)^{6|A|} \notag\\
	& \geq \prob_2(X_{\text{tiles}} =A) (1-p)^{6|A|} 				\notag\\
	& \geq \prob_2(X_{\text{tiles}} =A) e^{-7p|A|} \notag\\
	& \geq \prob_2(X_{\text{tiles}} = A) e^{-7hmp^{5/2}}.
		&&\hspace{-26pt}\text{(since $|A| \leq hm p^{3/2}$)} \label{E:P1andP2}
\end{align}

Let $\mathcal{C} = \{A: A \subseteq R \text{ and $R$ is traversable by $A$}\}$.  Then,
\begin{align*}
\prob(X_1 \in \mathcal{C})
	&\geq \prob(X_1^- \in \mathcal{C}\cap \mathcal{Q}_1)\\
	&=\sum_{A_1 \in \mathcal{Q}_1 \cap \mathcal{C}} \prob(X_1^- =A_1)\\
	&\geq \sum_{A \in \mathcal{Q} \cap \mathcal{C}} \prob_2(X_{\text{tiles}} = A) e^{-7hmp^{5/2}}		\qquad \text{(by inequality \eqref{E:P1andP2})}\\
	&=e^{-7hmp^{5/2}} \prob_2(X_{\text{tiles}} \in \mathcal{C} \cap \mathcal{Q})\\
	&=e^{-7hmp^{5/2}}\left[\prob_2(X_{\text{tiles}} \in \mathcal{C}) - \prob(X_{\text{tiles}} \in \mathcal{C} \setminus \mathcal{Q})\right].
\end{align*}

By Lemma \ref{L:crossing}, with $u = (1-p^2)^h$,  the probability of traversing satisfies $\prob_2(X_{\text{tiles}} \in \mathcal{C}) \geq \beta(u)^{m}(1-u^8)$  and so it remains to find an appropriate upper bound for $\prob(X_{\text{tiles}} \in \mathcal{C} \setminus \mathcal{Q})$.  First, 
\begin{align*}
\prob_2(X_{\text{tiles}} &\in \mathcal{C} \setminus \mathcal{Q})\\
	& \leq \prob_2(X_{\text{tiles}} \in \mathcal{C} \text{ and $X_{\text{tiles}}$ contains a triple})\\
	&\qquad + \prob_2(|X_{\text{tiles}}| > hmp^{3/2})\\
	& \leq \sum_{T \text{ a triple}}\prob_2(X_{\text{tiles}} \in \mathcal{C} \text{ and } T \subseteq X_{\text{tiles}}) + e^{-hmp^{3/2}} \beta(u)^m.
\end{align*}
Fix a triple $T$ and consider $\prob_2(X_{\text{tiles}} \in \mathcal{C} \text{ and } T \subseteq X_{\text{tiles}})$.  Note that the sites in the triple $T$ are contained in at most $4$ different columns of $R$.  Removing the columns containing sites from $T$ produces two smaller rectangles $R_1$ and $R_2$ both of height $h$.  If $X_{\text{tiles}}$ crosses $R$, there are rectangles $R_1'$ and $R_2'$ obtained from $R_1$ and $R_2$ by removing at most one column from each so that $R_1'$ and $R_2'$ are of height $h$ and of width $m_1$ and $m_2$ (respectively) with $m_1 + m_2 \geq m-6$ and with the property that $R'_1$ is traversable by $X_{\text{tiles}} \cap R'_1$ and $R_2'$ is traversable by $X_{\text{tiles}} \cap R'_2$.  Thus,
\begin{align*}
\prob_2(R &\text{ is traversable by $X_{\text{tiles}}$ and } T \subseteq X_{\text{tiles}})\\
	&\leq \prob_2(R'_1 \text{ is trav. by } X_{\text{tiles}}\cap R'_1 ) \prob_2(R'_2 \text{ is trav. by } X_{\text{tiles}}\cap R'_2 )\\
	& \qquad \cdot \prob_2(T \subseteq X_{\text{tiles}})\\
	&\leq \beta(u)^{m_1+m_2 - 2} (p^2)^2\\
	&\leq \beta(u)^{m-8} p^4.
\end{align*}
For each site $\mathbf{x} \in R$, consider the possible number of triples that contain $\mathbf{x}$ as one of the anchor sites.  There are $4$ different $2$-tiles that contain $\mathbf{x}$ as the anchor site.  If $\{\mathbf{x}, \mathbf{x}_2\}$ is a $2$-tile, then $|B_1^*(\{\mathbf{x}, \mathbf{x}_2\})| \leq 14$ and for each of the sites $\mathbf{y} \in B_1^*(\{\mathbf{x}, \mathbf{x}_2\})$, there are at $8$ different $2$-tiles that contain $\mathbf{y}$.  Since $\mathbf{x}$ could be the anchor site of one of two $2$-tiles, the number of triples that contain $\mathbf{x}$ as one of the anchor sites is at most $4\cdot 14\cdot 8/2 = 224$.  Thus 
\[
\sum_{T \text{ a triple}}\prob_2(R \text{ is traversable by $X_{\text{tiles}}$ and } T \subseteq X_{\text{tiles}}) \leq 224 hmp^4 \beta(u)^{m-8}
\]
 and so 
\begin{align*}
\prob_2(X_{\text{tiles}} \in \mathcal{C} \setminus \mathcal{Q})
	&\leq e^{-hmp^{3/2}} \beta(u)^m + 224hmp^4 \beta(u)^{m-8}\\
	& = \beta(u)^m(e^{-hmp^{3/2}} + 224hmp^4 \beta(u)^{-8}). 
\end{align*}
Since
\[
\beta(u) \geq 1-e^{-11p^2h} \geq 
	\begin{cases}
		(1-e^{-11})p^2h	&p^{1/8} \leq p^2h \leq 1\\
		1-e^{-11}		&1 \leq p^2h
	\end{cases}
\]
It follows that
$\beta(u)^{-8} \leq (1-e^{-11})^{-8}p^{-1}$.


Thus,
\begin{align*}
e^{-hmp^{3/2}} &+ 224hmp^4 \beta(u)^{-8}\\
	& \leq e^{-hmp^{3/2}} + 224hmp^4 (1-e^{-11})^{-8}p^{-1}\\
	& \leq e^{-hmp^{3/2}} + 225hmp^3\\
	& \leq e^{-p^{-3/8}} + 225 hmp^3 &&\hspace{-30pt}\text{(since $hm \geq p^{-15/8}$)}\\
	& \leq p^{9/8} + 225 hmp^3\\
	& \leq 226 hmp^3.	&&\hspace{-30pt}\text{(since $hm \geq p^{-15/8}$)}
\end{align*}

Combining these bounds yields
\begin{align*}
\prob(X_1 \in \mathcal{C})
	&\geq e^{-7hmp^{5/2}}(1-u^8 - 226hmp^3)\beta(u)^m\\
	&\geq e^{-7hmp^{5/2}}(1-u^8)\left(1-\frac{226hmp^3}{1-u^8}\right)\beta(u)^m.
\end{align*}

Now, by considering separately the two cases $p^{1/8} \leq p^2 h \leq 1$ and $1 \leq p^2 h \leq p^{-1/8}$,
\begin{align*}
\frac{226hmp^3}{1-u^8} 
	&\leq \frac{226hmp^3}{1-e^{-8p^2h}}\\
	&\leq 227 p^{5/8}.	&&\hspace{-20pt}\text{(since $m\leq p^{-1/4}$ and $p^2h \leq p^{-1/8}$)}	
\end{align*}
and so $1-226hmp^3/(1-u^8) \geq 1-227p^{5/8} \geq e^{-454p^{5/8}}$ for $p$ sufficiently small.  Thus 
\begin{align*}
\prob(X_1 \in \mathcal{C})
	&\geq e^{-7hmp^{5/2}}(1-u^8)e^{-454p^{5/8}}\beta(u)^m	\\
	&\geq e^{-7hmp^{5/2} - 454hm p^{15/8}p^{5/8}} (1-u^8) \beta(u)^m\\
	&\geq e^{-463hmp^{5/2}}(1-u^8) \beta(u)^m.
\end{align*}

Therefore, for $p$ sufficiently small, 
\begin{align*}
\prob(R \text{ is horiz. trav. by } X_1) &\geq e^{-463hmp^{5/2}}(1-u^8)\beta((1-p^2)^h)^m\\
		& =  e^{-463hmp^{5/2}}(1-u^8)e^{-mg(q^2h)}
\end{align*}
yielding the desired lower bound.
\end{proof}

\subsection{Lower bound on probability of spanning}

In the previous section, a bound on the crossing probability of a rectangle is given in terms of the function $\beta$ (or equivalently $g$). This is used in Lemma \ref{L:span_small_L}, below, to establish a bound on the probability that a large, but not arbitrarily large, rectangle is internally spanned.  Recall that $I(n, p)$ denotes the probability $[n]\times[n]$ is internally spanned when vertices are initially infected with probability $p$.

%

\begin{lemma}\label{L:span_small_L}
There exists a $p_1 >0$ such that if $p<p_1$, then 
\[
I(\lfloor p^{-17/8} \rfloor, p) \geq \exp\left(-\frac{2 \lambda + 2p^{1/9}}{p^2}\right).
\]
\end{lemma}

\begin{proof}
Fix $p \in (0,1)$ and set $m= \lfloor p^{-1/4} \rfloor$ and let $h_0$ be the smallest integer in $[p^{-15/8}, 2p^{-15/8}]$ such that $\lfloor p^{-17/8} \rfloor - h_0$ is divisible by $m$.  Set 
\[
n = (\lfloor p^{-17/8} \rfloor - h_0)/m
\]
 and for $j = 1, 2, \ldots, n$, set $h_j = j\cdot m + h_0$.  In particular $h_n = \lfloor p^{-17/8} \rfloor$ and $p^{-15/8}(1-3p^{1/4}) \leq n \leq p^{-15/8}$.

Setting $N = \lfloor p^{-17/8} \rfloor$, the square $[N]^2$ is internally spanned if the following three events all occur:
\begin{itemize}
	\item The sites $(1,1), (2,2), \ldots, (h_0, h_0)$, and $(1,2)$ are initially infected,
	\item for $j = 1, 2, \ldots, n-1$, the rectangles $[h_{j}+1, h_{j+1}] \times [1, h_j]$ are horizontally traversable from left to right and the rectangles $[1, h_j] \times [h_j +1, h_{j+1}]$ are vertically traversable from bottom to top, and
	\item for each $j = 1, 2, \ldots, n$, the rectangle $\{h_j\} \times [1, h_{j-1}]$ and the rectangle $[1, h_{j-1}] \times \{h_j\}$ each contain two adjacent infected sites.
\end{itemize}

Let $S$ denote the intersection of these three events. Note that for the third event
\[
\prob(\{h_{j+1}\} \times [1, h_j] \text{ contains two adjacent infected sites}) \geq 1-(1-p^2)^{(h_j-1)/2}.
\]

Since $S$ is the intersection of increasing events, by Lemma \ref{L:crossing_P1} and Harris's Lemma (inequality \eqref{E:FKG_Harris}),
\begin{align*}
\prob(S)		&\geq p^{h_0 + 1} \bigg(\prod_{j=0}^{n-1}\prob([h_j+1, h_{j+1}] \times [1, h_j]\text{ is trav. by } X_0)\\
			&\qquad (1-(1-p^2)^{(h_j-1)/2})\bigg)^2\\
			&\geq p^{h_0 + 1}  \bigg(\prod_{j=0}^{n-1} e^{-463h_j mp^{5/2}}(1-(1-p^2)^{8h_j})e^{-mg(q^2 h_j)}\\
			&\qquad (1-(1-p^2)^{(h_j-1)/2})\bigg)^2\\
			&= p^{h_0+1}  \left(\prod_{j=0}^{n-1}e^{-463h_j mp^{5/2}}(1-e^{-8q^2 h_j})e^{-mg(q^2 h_j)}(1-e^{-q^2(h_j-1)/2})\right)^2.
\end{align*}

The terms occurring in the above expression are simplified separately.  First, since $m=h_j-h_{j-1}$, and $q \geq p$,
\begin{align*}
\sum_{j=0}^{n-1} m g(q^2 h_j)
	& = \frac{1}{q^2} \sum_{j=0}^{n-1}mq^2 g(q^2 h_j)\\
	&\leq \frac{1}{p^2} \int_{0}^{\infty} g(x)\ dx =\frac{\lambda}{p^2}.
\end{align*}
Similarly, using the fact that $p^2 \leq q^2 \leq 2p^2$ and $p^{-1/4}/2 \leq m \leq p^{-1/4}$,
\begin{align*}
\sum_{j=0}^{n-1}{ h_j mp^{5/2}}
	&\leq \sum_{j=0}^{n-1}{ h_j p^{9/4}}	\\
	&\leq  n h_n p^{9/4}\\
	&\leq  p^{-15/8}p^{-17/8}p^{9/4}\\
	&=\frac{p^{1/4}}{p^2},\\
\sum_{j=0}^{n-1} -\log (1-e^{-8q^2 h_j})
	&\leq \frac{1}{8mq^2} \int_{0}^{\infty}\left( -\log(1-e^{-x})\right)\ dx	\\
	&\leq \frac{p^{1/4} (\pi^2 /24)}{p^2}, \text{ and}\\
\sum_{j=0}^{n-1} -\log(1-e^{-q^2(h_j - 1)/2})
	&\leq \frac{2}{q^2m} \int_0^{\infty} \left(-\log(1-e^{-x})\right)\ dx\\
	&\leq \frac{p^{1/4} (2 \pi^2 /3)}{p^2}.
\end{align*}
Finally,
\begin{align*}
p^{h_0 + 1}	&= \exp((h_0+1)\log p)\\
			&\geq \exp((p^{-15/8} + 1)\log p)\\
			& = \exp \left(-\frac{(-p^{1/8}\log p)(1+p^{15/8})}{p^2}\right)\\
			&\geq \exp \left(-\frac{p^{1/9}}{p^2} \right).		&&\text{(for $p$ small enough)}
\end{align*}

Combining these yields
\begin{align*}
I(\lfloor p^{-17/8} \rfloor, p) &\geq \prob(S)\\
		&\geq \exp\left(-\frac{(p^{1/9} + 2(463p^{1/4} +p^{1/4} \pi^2 /24 +p^{1/4} 2 \pi^2 /3 + \lambda)}{p^2} \right)\\
			&\geq \exp \left(-\frac{2p^{1/9}+2 \lambda}{p^2} \right),
\end{align*}
completing the proof of the lemma.

\end{proof}

The bound from Lemma \ref{L:span_small_L} can be further extended to an estimate of the probability that an arbitrarily large rectangle is internally spanned.

\begin{lemma}\label{L:span_big_L}
There is a $p_2 >0$ such that if $p<p_2$ and $n > p^{-17/8}$, 
\[
I(n,p) \geq \exp\left(-\frac{(2 \lambda + 3p^{1/9})}{p^2} \right).
\]
\end{lemma}

\begin{proof}
The main idea of the proof is that the grid $[n]^2$ is internally spanned if  the sub-square $[1,\lfloor p^{-17/8}\rfloor]^2$ is internally spanned and the rest of the grid contains many rows and columns with pairs of adjacent, initially infected sites that allow the infection to spread  one row and column at a time from this sub-square.

In particular, $[n]^2$ is internally spanned if the following events occur
\begin{itemize}
	\item the square $[\lfloor p^{-17/8}\rfloor]^2$ is internally spanned, and
	\item for each $j = \lfloor p^{-17/8} \rfloor + 1, \ldots, n$, the rectangles $\{j\} \times [1,j-1]$ and $[1, j-1] \times \{j\}$ both contains pairs of adjacent sites that are initially infected.
\end{itemize}
Let $S'$ be the above event and note that $S'$ is the intersection of many independent events.  

For each $j =  \lfloor p^{-17/8} \rfloor + 1, \ldots, n$, let $S_j$ be the event that $\{j\} \times [1,j-1]$ contains a pair of adjacent sites that are initially infected. Note that $\prob(S_j) = \prob([1,j-1] \times \{j\} \text{ contains a pair of adj. initially inf. sites})$ also.  Then, as in the proof of Lemma \ref{L:span_small_L},
$\prob(S_j) \geq 1- (1-p^2)^{\lfloor j/2 \rfloor}$
 and hence
\begin{align*}
\prob\left(\bigcap_{j=\lfloor p^{-17/8} \rfloor +1}^{n} S_j\right)^2
	&=\prod_{j=\lfloor p^{-17/8} \rfloor +1}^n \prob(S_j)^2\\
	&\geq \prod_{j=\lfloor p^{-17/8} \rfloor +1}^n (1- (1-p^2)^{\lfloor j/2 \rfloor})^2\\
	&= \exp \left( 2\sum_{j=\lfloor p^{-17/8} \rfloor +1}^{n} \log(1-e^{-q^2(j-1)/2}) \right)\\
	& \geq \exp \left(-\frac{2}{p^2} \int_{q^2(\lfloor p^{-17/8} \rfloor +1)/2}^{\infty}-\log (1-e^{-x})\ dx \right).
\end{align*}

It is straightforward to check that for every $k \geq 1$, the inequality
\[
 \int_{k}^{\infty}{\left(-\log(1-e^{-x})\right)\ dx} \leq \frac{5}{4}e^{-k}
 \]
 is satisfied.  Thus,
\begin{align*}
\prob(\cap_{j=\lfloor p^{-17/8} \rfloor +1}^{n} S_j)^2	&\geq \exp \left(-\frac{5/2e^{-q^2(\lfloor p^{-17/8} \rfloor +1)/2}}{p^2}\right)\\
		&\geq \exp \left(-\frac{5/2e^{-p^{-1/8}/4}}{p^2}\right)\\
		&\geq \exp\left(-\frac{p^{1/9}}{p^2}\right)
\end{align*}
and so $I(n, p) \geq \exp\left(-\frac{p^{1/9}}{p^2}\right) I(\lfloor p^{-17/8} \rfloor, p)$ and by the previous lemma,
\[
I(n,p)
	\geq \exp\left(-\frac{p^{1/9}}{p^2}\right) \exp\left(-\frac{2p^{1/9} + 2 \lambda}{p^2}\right) = \exp\left(-\frac{3p^{1/9} + 2\lambda}{p^2}\right)
\]
as claimed.
\end{proof}

Following an argument similar to that used by Holroyd \cite{aH03} for the analysis of the usual bootstrap process, Lemma \ref{L:span_big_L} is used to show that if $p^2 \log n > \lambda$, then $I(n,p)$ is close to $1$.

\begin{theorem}\label{T:lb}
For every $\varepsilon>0$, there exists $n_0 \in \mathbb{Z}^+$ such that if $n \geq n_0$ and $p \in (0,1)$ with $p \geq \sqrt{\frac{\lambda + \varepsilon}{\log n}}$ then
\[
I(n,p) \geq 1-3\exp(-n^{\varepsilon/6}).
\]
\end{theorem}

\begin{proof}
Fix $\varepsilon>0$ and $n_0 \geq 0$ large enough so that Lemma \ref{L:span_big_L}  applies for any $p$ with $p \leq \sqrt{\frac{\lambda+\varepsilon/2}{\log n_0}}$.

Fix $n \geq n_0$ and $p\in (0,1)$ with $p \geq \sqrt{\frac{\lambda + \varepsilon}{\log n}}$.  Note that by coupling, if $p' < p$ then $I(n,p') \leq I(n,p)$ and so it suffices to prove the claimed bound for $p=\sqrt{\frac{\lambda + \varepsilon}{\log n}}$.

Instead of randomly infecting all sites at once, sites are infected in two `rounds'.  Two random configurations of infected sites are independently coupled so that a large sub-rectangle of $[n]^2$ is likely to be internally spanned by sites from the first configuration and that, using only sites from the second configuration, the infection is able to spread row by row and column by column from this rectangle to the entire grid.

Set $p_1 = \sqrt{\frac{\lambda+\varepsilon/2}{\log n}}$ and $p_2 = \frac{\varepsilon/2}{\log n}$.  Define one set $X_0 \sim \operatorname{Bin}([n]^2, p)$.  Let $X_0' \sim \operatorname{Bin}([n]^2, p_1)$ and $X_0'' \sim \operatorname{Bin}([n]^2, p_2)$ be coupled with $X_0$ so that $X_0' \cup X_0'' \subseteq X_0$.  This is possible since for $n \geq 75$, $p_1 + (1-p_1)p_2 \leq p$.

Set $\ell = \left\lfloor \exp\left(\frac{\varepsilon}{8p_1^2}\right)\right\rfloor$.  Note that since $\lambda < 1/8$,
\[
\ell = \left\lfloor \exp\left(\frac{\varepsilon}{8p_1^2}\right)\right\rfloor
	=\lfloor n^{\frac{\varepsilon}{8(\lambda + \varepsilon/2)}} \rfloor
	\leq n^{\varepsilon}
	<n.
\]
Divide the grid $[n]^2$ into $\lfloor n/\ell \rfloor^2$ disjoint $\ell \times \ell$ sub-grids, with potentially some remainder: $\{[k\ell +1, (k+1)\ell] \times [j\ell +1, (j+1)\ell]:\ k,j \in [0, \lfloor n/\ell \rfloor -1]\}$.  For each of these $\ell \times \ell$ sub-grids, the probability that the sub-grid is internally spanned by $X_0'$ is $I(\ell, p_1)$.  The probability that none of these $\ell \times \ell$ sub-grids are internally spanned is
\begin{align*}
(1-I(\ell, p_1))^{\lfloor n/\ell \rfloor^2}
	&\leq (1-I(\ell, p_1))^{\frac{n^2}{2\ell^2}}\\
	&\leq \exp \left(-\frac{n^2}{2\ell^2}I(\ell, p_1) \right).
\end{align*}  
Now,
\begin{align*}
\frac{n^2}{2\ell^2} I(\ell, p_1)
	&\geq \frac{n^2}{2n^{\varepsilon/8\lambda}}\exp\left(-\frac{2\lambda + 3p_1^{1/4}}{p_1^2} \right)
			&&\text{(by Lemma \ref{L:span_big_L})}\\
	&\geq \frac{1}{2}n^{2-\frac{\varepsilon}{8\lambda}}\exp\left(-\frac{2\lambda + 3p_1^{1/4}}{\lambda+ \varepsilon/2}\log n\right)\\
	&\geq \frac{1}{2}n^{2-\frac{\varepsilon}{8\lambda}}\exp\left(-\left(2-\frac{\varepsilon}{\lambda}\right)\log n\right)\\
	& = \frac{1}{2}n^{2-\frac{\varepsilon}{8\lambda}}n^{-2 + \frac{\varepsilon}{\lambda}}\\
	&= \frac{1}{2}n^{\frac{7 \varepsilon}{8\lambda}}\\
	&\geq n^{\frac{3\varepsilon}{4\lambda}}.
			&&\left(\text{for $n \geq \exp\left(\frac{1}{2\varepsilon} \right)$}\right)
\end{align*}
Let $S$ be the event that at least one $\ell \times \ell$ sub-grid is internally spanned by $X_0'$.  Then, since $\lambda \leq 1/12$,
\begin{equation}\label{E:existslbyl}
\prob(S) \geq 1-\exp(-n^{9\varepsilon}).
\end{equation}

Next, consider the probability that an internally spanned $\ell \times \ell$ sub-grid, together with sites in $X_0''$ will percolate in $[n]^2$.  As in Lemma \ref{L:span_big_L} the probability of this occurring is bounded below by the probability that, in many rows and columns, there are pairs of adjacent infected sites.

Let $A_r$ be the event that for every $k$ and $j$ with $0\leq k \leq \lfloor n/\ell \rfloor-1$ and $1 \leq j \leq n$, the row $[k\ell+1, (k+1)\ell] \times \{j\}$ contains at least two adjacent infected sites in $X_0''$.  Then
\begin{align*}
\prob(A_r) \geq (1-(1-p_2^2)^{(\ell-1)/2})^{n\lfloor n/\ell \rfloor}
	&\geq (1-\exp(-p_2^2(\ell-1)/2))^{n^2/\ell}\\
	&\geq \exp\left(-\frac{2n^2}{\ell} e^{-p_2^2 \ell/3}\right).
\end{align*}
Now, for $n$ large enough, $(\log n)^2 \leq \frac{\varepsilon^2}{12} n^{\frac{\varepsilon}{72(\lambda + \varepsilon/2)}}$ and so
\[
\frac{p_2^2 \ell}{3} = \frac{\varepsilon^2 \lfloor n^{\varepsilon/8(\lambda+\varepsilon/2)} \rfloor}{12(\log n)^2} \geq n^{\frac{\varepsilon}{9(\lambda + \varepsilon/2)}}.
\]
Similarly, for $n$ sufficiently large, depending on $\varepsilon$, 
\[
2n^{2-\frac{\varepsilon}{8(\lambda+\varepsilon/2)}}\exp(-n^{\frac{\varepsilon}{9(\lambda + \varepsilon/2)}}) \leq \exp(-n^{\frac{\varepsilon}{10(\lambda + \varepsilon/2)}}) \leq \exp(-n^{\varepsilon/6}),
\]
and hence 
\begin{equation}\label{E:allrows}
\prob(A_r) \geq \exp(-e^{-n^{\varepsilon/6}}).
\end{equation}

Similarly, define $A_c$ to be the event that for every $k$ and $j$ with $0\leq k \leq \lfloor n/\ell \rfloor-1$ and $1 \leq j \leq n$, the column $\{j\} \times [k\ell+1, (k+1)\ell]$ contains at least two adjacent infected sites in $X_0''$.  Then $\prob(A_c) = \prob(A_r)$ and since the events $A_c$ and $A_r$ are both increasing events, by Harris's Lemma (inequality \eqref{E:FKG_Harris}), $\prob(A_c \cap A_r) \geq \prob(A_c)\prob(A_r)$.  Now, if both events $S$ and $A_c \cap A_r$ occur, then $[n]^2$ is internally spanned by the set of initially infected sites $X_0' \cup X_0''$.  Thus,
\begin{align*}
I(n,p)	&\geq \prob(S) \prob(A_c \cap A_r)\\
		&\geq \prob(S) \prob(A_c) \prob(A_r)\\
		&\geq (1-\exp(-n^{9\varepsilon})) \exp(-2e^{-n^{\varepsilon/6}})
			&&\text{(by eqns. \eqref{E:existslbyl} and \eqref{E:allrows})}\\
		&\geq 1-2\exp(-n^{\varepsilon/6}) - \exp(-n^{9\varepsilon})\\
		&\geq 1-3\exp(-n^{\varepsilon/6}).
\end{align*}
For $n$ sufficiently large, depending on $\varepsilon$ and if $p \geq \sqrt{\frac{\lambda+\varepsilon}{\log n}}$, then  $I(n,p) \geq 1-3\exp(-n^{\varepsilon/6})$.
\end{proof}

In particular, for every $\varepsilon>0$ and any sequence $\{p(n)\}_{n \in \mathbb{N}} \subseteq (0,1)$ with the property that for all $n \in \mathbb{N}$,  $p(n) \geq \sqrt{\frac{\lambda + \varepsilon}{\log n}}$, then
\[
I(n,p(n)) \geq 1-3\exp(-n^{\varepsilon/6}) = 1-o(1)
\]
and so with high probability, a random set of initially infected sites $X_0 \sim \operatorname{Bin}([n]^2, p(n))$ percolates in the recovery bootstrap process.  Thus the critical probability satisfies
\[
p_c([n]^2, \mathcal{R}) \leq \sqrt{\frac{\lambda + o(1)}{\log n}}.
\]

\section{Upper bound for probability of percolation}\label{S:ub}

\subsection{Traversing rectangles and growing rectangles}

In the previous section, some infected sites were omitted from any set of initially infected sites to produce the set $X^-$ that could be compared to the scheme of infection with $2$-tiles.  To obtain an upper bound for the probability of percolation in the recovery bootstrap process, an alteration of the initial configuration is defined that is different from the one given in Section \ref{S:lb}.   An initial configuration of infected sites $X$ is altered to produce a new configuration $X^+$ that can be more easily compared to the process of infecting sites with $2$-tiles, but in such a way that if $X$ percolates in $\mathcal{R}$, then so does $X^+$.  The idea is to uninfect isolated sites that do not affect the final infection status of any of their neighbours, while including some new infected sites next to isolated sites that have a chance of affecting whether or not their neighbours become infected.


For convenience, denote $e_1 = (1,0)$, $e_2 = (0,1)$, $e_3 = (-1,0)$ and $e_4 = (0,-1)$.  Recall the definitions of two different types of distances on the grid: balls in the $\ell_{\infty}$ metric are written $B_r^*(\mathbf{x})$ while balls in the $\ell_1$ metric are written $B_1(\mathbf{x})$ (see \ref{D:ell1ball} and \ref{D:ellmaxball}).

\begin{definition}
For any $X \subseteq \mathbb{Z}^2$, define $X^+ \subseteq \mathbb{Z}^2$ as follows:
\begin{itemize}
	\item If $\mathbf{x}\in X$ with $B_1^*(\mathbf{x}) \cap X \neq \{\mathbf{x}\}$, then $\mathbf{x} \in X^+$.
	\item If $\mathbf{x} \in X$ with $B_2(\mathbf{x}) \cap X = \{\mathbf{x}\}$ then $\mathbf{x} \notin X^+$.
	\item If $\mathbf{x} \in X$ is isolated and for some $i \in \{1, 2, 3, 4\}$, $\mathbf{x} + 2e_i \in X$, then 
		\begin{itemize}
			\item if $B_2(\{\mathbf{x}, \mathbf{x}+e_i, \mathbf{x} +2e_i\}) \cap X \setminus \{\mathbf{x}, \mathbf{x}+2e_i\} = \emptyset$ then $\mathbf{x} \notin X^+$, and 
			\item  if $B_2(\{\mathbf{x}, \mathbf{x}+e_i, \mathbf{x} +2e_i\}) \cap X \setminus \{\mathbf{x}, \mathbf{x}+2e_i\} \neq \emptyset$ then $\mathbf{x}, \mathbf{x}+e_i \in X^+$.
		\end{itemize}
\end{itemize}
\end{definition}

Figure \ref{F:X+uninfect} shows the configurations of  infected sites in $X$ that are uninfected in $X^+$.  The shaded sites represent infected sites and uninfected sites are represented by sites containing empty circles.

\begin{figure}[h]
	\begin{center}
	\begin{tikzpicture}[line width=1pt, scale = 0.6]
	\draw[step=0.8cm] (-0.4, -0.4) grid (4.4, 4.4);
		\filldraw[fill=black] (1.6, 1.6) rectangle (2.4, 2.4);
		
		\draw (0.4,2) circle (0.1cm); 
		\draw (1.2,1.2) circle (0.1cm);
		\draw (1.2,2) circle (0.1cm); 
		\draw (1.2,2.8) circle (0.1cm);
		\draw (2, 0.4) circle (0.1cm);
		\draw (2,1.2) circle (0.1cm);
		\draw (2,2.8) circle (0.1cm); 
		\draw (2,3.6) circle (0.1cm);
		\draw (2.8,1.2) circle (0.1cm);
		\draw (2.8,2) circle (0.1cm); 
		\draw (2.8,2.8) circle (0.1cm);
		\draw (3.6, 2) circle (0.1cm);
	
	\draw[step=0.8cm] (5.2, -0.4) grid (11.6, 4.4);
	\filldraw[fill=black] (7.2, 1.6) rectangle (8, 2.4);
	\filldraw[fill=black] (8.8, 1.6) rectangle (9.6, 2.4);
	
		\draw (6,2) circle (0.1cm); 
		\draw (6.8,1.2) circle (0.1cm);
		\draw (6.8,2) circle (0.1cm); 
		\draw (6.8,2.8) circle (0.1cm);
		\draw (7.6, 0.4) circle (0.1cm);
		\draw (7.6,1.2) circle (0.1cm);
		\draw (7.6,2.8) circle (0.1cm); 
		\draw (7.6,3.6) circle (0.1cm);
		\draw (8.4, 0.4) circle (0.1cm);
		\draw (8.4,1.2) circle (0.1cm);
		\draw (8.4,2) circle (0.1cm);
		\draw (8.4,2.8) circle (0.1cm); 
		\draw (8.4,3.6) circle (0.1cm);
		\draw (9.2, 0.4) circle (0.1cm);
		\draw (9.2,1.2) circle (0.1cm);
		\draw (9.2,2.8) circle (0.1cm); 
		\draw (9.2,3.6) circle (0.1cm);
		\draw (10,1.2) circle (0.1cm);
		\draw (10,2) circle (0.1cm); 
		\draw (10,2.8) circle (0.1cm);
		\draw (10.8,2) circle (0.1cm);

\end{tikzpicture}

	\end{center}
		\caption{Sites from $X$ that are uninfected}
		\label{F:X+uninfect}
\end{figure}
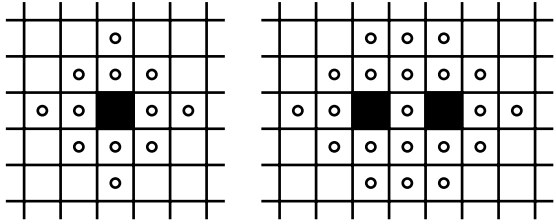

Figure \ref{F:X+infect} shows an isolated site $\mathbf{x}$ with $\mathbf{x}, \mathbf{x}+2e_1 \in X$.  If any
other site inside the outlined region is infected (in $X$), then $\mathbf{x}$ and $\mathbf{x}+e_1$ (the site containing a shaded circle) are included in $X^+$.

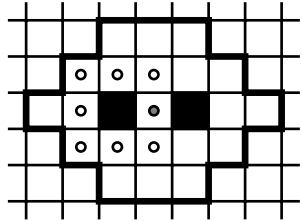
\begin{figure}[h]
	\begin{center}
	\begin{tikzpicture}[line width=1pt, scale = 0.6]
\draw[step=0.8cm] (5.2, -0.4) grid (11.6, 4.4);
	\filldraw[fill=black] (7.2, 1.6) rectangle (8, 2.4);
	\filldraw[fill=black] (8.8, 1.6) rectangle (9.6, 2.4);
	
		\draw (6.8,1.2) circle (0.1cm);
		\draw (6.8,2) circle (0.1cm); 
		\draw (6.8,2.8) circle (0.1cm);
		\draw (7.6,1.2) circle (0.1cm);
		\draw (7.6,2.8) circle (0.1cm); 
		\draw (8.4,1.2) circle (0.1cm);
		\filldraw[fill=gray] (8.4,2) circle (0.1cm);
		\draw (8.4,2.8) circle (0.1cm); 

		\draw[line width=2.5pt] (5.6, 1.6) 	-- ++(0,0.8) 
								-- ++(0.8,0)
								-- ++(0,0.8) 
								-- ++(0.8,0)
								-- ++(0,0.8) 
								-- ++(2.4,0) 
								-- ++(0,-0.8)
								-- ++(0.8, 0)
								-- ++(0,-0.8)
								-- ++(0.8,0)
								-- ++(0, -0.8)
								-- ++(-0.8, 0)
								-- ++(0, -0.8)
								-- ++(-0.8, 0)
								-- ++(0, -0.8)
								-- ++(-2.4, 0)
								-- ++(0, 0.8)
								-- ++(-0.8, 0)
								-- ++(0, 0.8)
								-- cycle;
\end{tikzpicture}
\end{center}
		\caption{Sites included in $X^+$}
		\label{F:X+infect}
\end{figure}
Call any such configuration of three infected sites in $X$ a \emph{triplet}.  In Figure \ref{F:triplets}, the different types of triplets are shown with the associated sites marked with an empty circle.  Considering rotations and reflections, there are $2$ triplets of each of the first and second type, $8$ triplets of each of the third and fourth type, and $4$ triplets of each of the fifth and sixth types.  Thus, in total, there are $28$ different triplets.

\begin{figure}[h]
\begin{center}
\begin{tikzpicture}[line width=1pt, scale = 0.8]
	\draw[step=0.6cm] (-0.3, -0.3) grid ++(1.2, 3.6);
		\filldraw[fill=black, opacity=0.9] (0,0) rectangle ++(0.6, 0.6);
		\filldraw[fill=black, opacity=0.9] (0,1.2) rectangle ++(0.6, 0.6);
		\filldraw[fill=black, opacity=0.9] (0,2.4) rectangle ++(0.6, 0.6);
		\draw (0.3,0.9) circle (0.075cm);
		\draw (0.3,2.1) circle (0.075cm);
		
	\draw[step=0.6cm] (1.5, 0.3) grid ++(1.2, 3);
		\filldraw[fill=black, opacity=0.9] (1.8,0.6) rectangle ++(0.6, 0.6);
		\filldraw[fill=black, opacity=0.9] (1.8,1.2) rectangle ++(0.6, 0.6);
		\filldraw[fill=black, opacity=0.9] (1.8,2.4) rectangle ++(0.6, 0.6);
		\draw (2.1,2.1) circle (0.075cm);

	\draw[step=0.6cm] (3.3, 0.3) grid ++(1.8, 3);
		\filldraw[fill=black, opacity=0.9] (3.6,0.6) rectangle ++(0.6, 0.6);
		\filldraw[fill=black, opacity=0.9] (4.2,1.2) rectangle ++(0.6, 0.6);
		\filldraw[fill=black, opacity=0.9] (4.2,2.4) rectangle ++(0.6, 0.6);
		\draw (4.5,2.1) circle (0.075cm);
	
	\draw[step=0.6cm] (5.7, 0.9) grid ++(1.8, 2.4);
		\filldraw[fill=black, opacity=0.9] (6,1.2) rectangle ++(0.6, 0.6);
		\filldraw[fill=black, opacity=0.9] (6.6,1.2) rectangle ++(0.6, 0.6);
		\filldraw[fill=black, opacity=0.9] (6.6,2.4) rectangle ++(0.6, 0.6);
		\draw (6.9,2.1) circle (0.075cm);
		
	\draw[step=0.6cm] (8.1, 0.9) grid ++(2.4, 2.4);
		\filldraw[fill=black, opacity=0.9] (8.4,1.2) rectangle ++(0.6, 0.6);
		\filldraw[fill=black, opacity=0.9] (8.4,2.4) rectangle ++(0.6, 0.6);
		\filldraw[fill=black, opacity=0.9] (9.6,1.2) rectangle ++(0.6, 0.6);
		\draw (8.7,2.1) circle (0.075cm);
		\draw (9.3,1.5) circle (0.075cm);
		
	\draw[step=0.6cm] (11.1, 0.9) grid ++(2.4, 2.4);
		\filldraw[fill=black, opacity=0.9] (11.4,1.2) rectangle ++(0.6, 0.6);
		\filldraw[fill=black, opacity=0.9] (12,2.4) rectangle ++(0.6, 0.6);
		\filldraw[fill=black, opacity=0.9] (12.6,1.2) rectangle ++(0.6, 0.6);
		\draw (12.3,1.5) circle (0.075cm);

\end{tikzpicture}

\end{center}
		\caption{Six types of triplets}
		\label{F:triplets}
\end{figure}
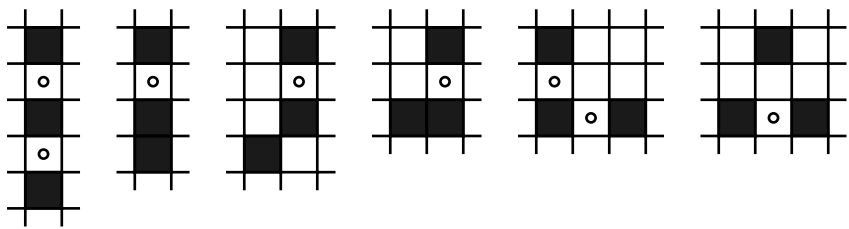

As before, pairs of infected points that form one of the four $2$-tiles are called a \emph{double}.  By definition, every site contained in a configuration $X^+$ is either contained in a double in $X$, or associated with a set of three points in $X$ that form a triplet.  Note that if $X$ percolates, then so does $X^+$ since any site in $X \setminus X^+$ either becomes uninfected in the first or second time step and does not contribute to the infection of any other sites.

In the next lemma, the probability that a rectangle $R$ is traversable by the set $X^+$ is compared to the probability that $R$ is traversable by a configuration on $2$-tiles.  As in Section \ref{S:lb}, first, rectangles with height close to $p^{-2}$ are considered.

\begin{lemma}\label{L:cross-upper}
Let $B >1$, $Z>0$, $m \in \mathbb{N}$ and set
\[
Q_1(B,Z,m) = 1500Bm + \frac{30Bm + \frac{3\cdot 25^3}{Z}}{(1-e^{-11Z})^{m-1}}.
\]
There exists $p_0 = p_0(B,Z,m)>0$ such that for all $h \in \mathbb{N}$ with $Z \leq hp^2 \leq B$ and every rectangle $R$ of dimension $(m,h)$, if $p< p_0$ and $X \sim \operatorname{Bin}(R,p)$ then
\[
\prob(R \text{ is horizontally traversable}) \leq (1 + pQ_1(B,Z,m)) e^{-g(hq^2)(m-1)}.
\]
\end{lemma}

\begin{proof}
Fix $p>0$, $B>1$, $Z>0$, $m \in \mathbb{N}$ and let $h \in \mathbb{N}$ be such that $Z \leq hp^2 \leq B$.  Let $R$ be a rectangle of dimension $(m,h)$. Following an approach similar to that used in Lemma \ref{L:crossing_P1}, let 
\begin{align*}
\mathcal{Q}&=\{A \subseteq R:\ \text{every site in $A$ has a $\ell_{\infty}$-nbr, $A$ contains no triplets and }\\
	&\hspace{30pt} |A| \leq |R|p\}
\end{align*} 
and let $\mathcal{C}$ be the collection of configurations of infected sites for which $R$ is horizontally traversable from  left to right in the process $\mathcal{R}$.  Fix $A \in \mathcal{Q}$ and let $X \sim \operatorname{Bin}(R,p)$.  Since $R$ will be horizontally traversable by $X^+$ if $R$ is horizontally traversable by $X$, 
\[
\prob(R \text{ is horiz. trav. by } X) \leq \prob(R \text{ is horiz. trav. by } X^+).
\]
  If $X^+ = A$, then since $A$ contains no triplets and any site in $X^+ \setminus X$ is contained in a triplet, $A \subseteq X$.  Further, any site in $X \setminus A$ is isolated and not contained in a triplet.  In order to deal with independent events, consider the following two events:
  \begin{itemize}
	\item $E_1$: every site in $A$ is in $X$, and 
	\item $E_2$: there are no doubles or triplets in $X\cap (R \setminus B_3(A))$.
\end{itemize}
Since $E_1$ and $E_2$ are independent, $\prob(X^+ = A) \leq \prob(E_1)\prob(E_2) = p^{|A|} \prob(E_2)$.  In order to bound $\prob(E_2)$, a version of Janson's inequality \eqref{E:Jansonineq} is used.  

Let $(B_i)_{i\in I}$ be the sequence of events that a particular double or triplet occurs in $X\cap (R \setminus B_3(A))$.  For each site $\mathbf{x}$, there are $4$ different doubles containing $\mathbf{x}$ as the left-most and bottom-most site and there are $28$ different triplets containing $\mathbf{x}$ as the left-most and bottom-most site.  Thus, there are at most $4|R \setminus B_3(A)|$ such doubles and at most $28|R \setminus B_3(A)|$ such triplets.  Consider the number of sites in $|B_3(A)|$.  For any double $\{\mathbf{x}_1, \mathbf{x}_2\}$, $|B_3(\{\mathbf{x}_1, \mathbf{x}_2\})| = 32 = 16|\{\mathbf{x}_1, \mathbf{x}_2\}|$.  Thus, since every site in the configuration $A$ is contained in a double, $|B_3(A)| \leq 16|A|$.

Then $E_2 = \cap_{i \in I}\overline{B}_i$ and this event depends only on the $|R \setminus B_3(A)| \geq |R|-16|A|$ independent events that a particular site in $R \setminus B_3(A)$ is initially infected or not.  In order to apply inequality \eqref{E:Jansonineq}, a bound is required for the sum of probabilities of events $B_i \cap B_j$ for which $B_i$ and $B_j$ are not independent.  

Consider the number of overlapping doubles and triples.  For each site $\mathbf{x}$, there are $4$ doubles containing $\mathbf{x}$ as the anchor and $2$ sites in the double that could be overlapping with another double.  For the sites in the first double, there are $8$ different doubles containing that site.  In this way each pair of overlapping doubles is counted twice  and so there are at most $32|R \setminus B_3(A)|$ different pairs of overlapping doubles. 

 Similarly, since there are $28$ different triplets, there are at most $4\cdot 2 \cdot 3 \cdot 28 = 672$ different pairs of a triple and an overlapping double that $\mathbf{x}$ as its anchor at at most $28 \cdot 28 \cdot 3^2/2 = 3528$ pairs of triples that contain $\mathbf{x}$ as the lowest left-most site of one of the triplets.  Therefore, in all, there are $672|R \setminus B_3(A)|$ different pairs of a double and a triplet that overlap and at most $3528|R \setminus B_3(A)|$ pairs of overlapping triplets.  Since a pair of distinct doubles that overlap contain at least $3$ sites, a double and a triple that overlap contain at least $3$ sites and a pair of distinct triplets that overlap contain at least $4$ sites,
\begin{align*}
\sum_{B_i, B_j \text{ not indep.}} \prob(B_i \cap B_j) 
	&\leq (32 + 672)|R \setminus B_3(A)|p^3 + 3528|R \setminus B_3(A)|p^4\\
	& \leq 710 |R \setminus B_3(A)|p^3
\end{align*}
when $p \leq 1/588$.  Similarly, $\sum_{i \in I}\prob(B_i) \leq (4p^2 + 28 p^3)|R \setminus B_3(A)|$ and applying Janson's inequality \eqref{E:Jansonineq},
\begin{align*}
\prob(E_2) &= \prob(\cap_{i\in I} \overline{B}_i)\\
	&\leq \exp(-(4p^2 + 28 p^3)|R \setminus B_3(A)|) \exp(710 |R \setminus B_3(A)|p^3)\\
	&=\exp((-4p^2 + 682 p^3)|R \setminus B_3(A)|)\\
	&\leq \exp((-4p^2 + 682 p^3)(|R|-16|A|))\\
	&\leq \exp(-4|R|p^2+64p^2|A| + 682p^3|R|)\\
	&\leq \exp(-4p^2|R|+ 746p^3|R|).
		&&\hspace{-100pt}\text{(since $|A| \leq |R|p$)}
\end{align*}
Thus, since the event that $X^+=A$ is contained in the intersection of independent events $E_1$ and $E_2$,
\begin{align*}
\prob&(X^+ = A)\\
	&\leq p^{|A|} \exp(-4p^2|R|+746|R|p^3) \\
	& = p^{|A|}(1-p^2)^{4|R| - |A|/2}(1-p^2)^{-4|R| + |A|/2} \exp(-4p^2|R|+746|R|p^3) \\
	& = \prob_2(X_{\text{tiles}} = A) (1-p^2)^{-4|R| + |A|/2} \exp(-4p^2|R|+746|R|p^3)\\
	&\leq \prob_2(X_{\text{tiles}} = A) (1-p^2)^{-4|R|} \exp(-4p^2|R|+746|R|p^3).
\end{align*}
For $p$ sufficiently small, $1-p^2 \geq e^{-(p^2+p^4)}$ and for $x$ small enough, $e^x \leq 1+2x$.  Thus,
\begin{align*}
(1-p^2)^{-4|R|} &\exp(-4p^2|R|+746|R|p^3)\\
	&\leq  \exp(4|R|(p^2 + p^4) -4p^2|R| + 746|R| p^3)\\
	&\leq \exp(750|R|p^3)\\
	&\leq \exp\left(750\frac{B}{p^2} m p^3\right)\\
	&= \exp(750Bmp)\\
	&\leq 1+ 1500 Bmp.
\end{align*}
Therefore,
\[
\prob(X^+ = A) 
	\leq \prob_2(X_{\text{tiles}} = A)(1+1500Bmp). 
\]
This inequality can be used to compare the probability that $R$ is traversable by $X^+$ to that of $R$ being traversable by a random configuration of $2$-tiles, conditioned on either configuration being in the collection $\mathcal{Q}$.  

Consider now the probability that $X^+ \notin \mathcal{Q}$.  Since every site in $X^+$ has a neighbour, if $X^+ \notin \mathcal{Q}$ then either $X^+$ contains a triplet or $|X^+| > |R|p$.  Let $\{T_j\}_{j \in J}$ be the collection of sets of sites in $R$ that form triplets and consider first the probability that $X^+$ contains one of the triplets $T_j$.  If $T_j \subseteq X^+$ then either $T_j \subseteq X$ or else one of the sites in $T_j$ is associated with another triplet contained in $X$.  In particular, if $T_j \nsubseteq X$, then every site in $T_j \setminus X$ is adjacent to at least $2$ sites in $X$ and together with sites in $T_j \cap X$, there are at least $4$ sites in $X$.  If a site $\mathbf{x} \in T_j \setminus X$ is associated with another triplet in $X$, then either $\{\mathbf{x} +(-1,0),\mathbf{x} +(1,0) \}\subseteq X$ or $\{\mathbf{x} +(0,-1),\mathbf{x} +(0,1)\} \subseteq X$.  Very roughly then $\prob(T_j \subseteq X^+) \leq p^3 + 3^3 p^4$.  Since there are at most $28 |R|$ different triplets in $R$ and $|R| = hm \leq Bm/p^2$,
\begin{equation}\label{E:X+triples}
\prob(\cup_{j \in J}\{T_j \subseteq X^+\}) \leq \sum_{j \in J} \prob(T_j \subseteq X^+) \leq 28|R|(p^3 +27p^4) \leq 30Bmp
\end{equation}
as long as $p \leq 1/378$.  

It is slightly more complicated to determine the probability that $|X^+| \geq |R|p$ since the events that any two sites are included in $X^+$ are not, in general, independent.  

Since the membership in $X^+$ of any site is determined by at most $25$ independent events, the initial infection of sites in $X$ within a ball of radius $3$, a version of Talagrand's inequality \cite{mT95}  can be used to bound the probability that $X^+$ is large.

For every site $\mathbf{x} \in R$, 
\[
\prob(\mathbf{x} \in X^+) \leq 8p^2 + 100p^3 \leq 9p^2
\]
 when $p$ is sufficiently small.  Thus $\expect(|X^+|) \leq 10|R|p^2$.  Changing the initial infection status of one site changes the value of $|X^+|$ by at most $25$ and for any $r$, the event that $|X^+|\geq r$  can be certified by the initial infection status of  $25r$ sites.  Thus, applying Talagrand's inequality (Theorem \ref{T:tal}),
\begin{align}
\prob(|X^+| \geq |R|p)	
	&\leq \exp \left(\frac{-(|R|p - 9|R|p^2 )^2}{2\cdot 25^3|R|p}\right) \notag\\
	&\leq \exp\left(-\frac{|R|p}{3\cdot 25^3}\right)
		&&\text{(for $p \leq 1/50$)} \notag\\
	&\leq \exp\left(-\frac{Zm}{3 \cdot 25^3 p}\right)	
		&&\text{(since $h \geq Z/p^2$)} \notag\\
	&\leq \frac{3 \cdot 25^3 p}{Zm}.
		&&\text{(using $e^{-x} \leq 1/x$)} \label{E:X+big}
\end{align}

Thus, the probability that $X^+$ is not a configuration in $\mathcal{Q}$ can be estimated as follows.  Combining the two inequalities \eqref{E:X+triples} and \eqref{E:X+big}, yields
\begin{align}
\prob(X^+ \notin \mathcal{Q})
	&\leq \prob(X^+ \text{ contains a triplet}) + \prob(|X^+| \geq |R|p) \notag\\
	&\leq 30Bmp + \frac{3 \cdot 25^3 p}{Zm}.\label{E:notQ}
\end{align}

Finally, it is possible to bound from above the probability that $R$ is horizontally traversable by $X^+$ using Lemma \ref{L:crossing},
\begin{align*}
\prob(X^+ \in \mathcal{C})
	&\leq \prob(X^+ \in \mathcal{C} \cap \mathcal{Q}) + \prob(X^+ \notin \mathcal{Q})\\
	&\leq \sum_{A \in \mathcal{C} \cap \mathcal{Q}}\prob(X^+ = A) + 30Bmp + \frac{3 \cdot 25^3 p}{Zm}\\
	&\leq \sum_{A \in \mathcal{C}\cap \mathcal{Q}}\prob_2(X_{\text{tiles}} = A)(1+1500Bmp) + 30Bmp + \frac{3 \cdot 25^3 p}{Zm}\\
	&\leq \prob(X_{\text{tiles}} \in \mathcal{C})(1+1500Bmp) + 30Bmp + \frac{3 \cdot 25^3 p}{Zm}\\
	&\leq e^{-g(hq^2)(m-1)}(1+1500Bmp) + 30Bmp + \frac{3 \cdot 25^3 p}{Zm}\\
	&= e^{-g(hq^2)(m-1)}\left(1+1500Bmp + \frac{30Bmp + \frac{3 \cdot 25^3 p}{Zm}}{(1-e^{-11hp^2})^{m-1}}\right)\\
	&\leq e^{-g(hq^2)(m-1)}\left(1+ p\left(1500Bm + \frac{30Bm + \frac{3\cdot 25^3}{Z}}{(1-e^{-11Z})^{m-1}}\right) \right)\\
	&=  e^{-g(hq^2)(m-1)}\left(1+pQ_1(B,Z,m)\right).
\end{align*}

Thus $\prob(R \text{ is horiz. trav. by } X) \leq e^{-g(hq^2)(m-1)}\left(1+pQ_1(B,Z,m)\right)$.
\end{proof}

In Lemma \ref{L:cross-upper}, the width of the rectangle being traversed is arbitrary.  However, when $B$ and $Z$ are fixed, $Q_1(B, Z, m)$ is increasing in $m$ and for large values of $m$, part of the error term, $Q_1(B,Z,m)$, might become too large for this lemma to be useful for upper bounds on the the probability of percolation.  Instead of considering the probability of traversing a large rectangle all at once, it is useful to consider traversing `strips' of a fixed width one at a time.  There can potentially be dependence between the probability of crossing adjacent strips, but this can be dealt with by ignoring the infection configuration in a few columns.  The following lemma gives the details.

\begin{lemma}\label{L:cross-strips}
Fix $B>1$, $Z >0$, $p<p_0$ and $h \in [Z/p^2, B/p^2]$.  For any $m,w  \in \mathbb{N}$ with $w<m$ and  any rectangle  $R$ of dimensions $(m,h)$,
\[
\prob(R \text{ is horiz. trav.}) \leq (1+ pQ_1(B, Z,w))^{m/w + 1}e^{-g(hq^2) m(1-11/w)}.
\]
\end{lemma}

\begin{proof}
Fix $w <m$ and a rectangle $R$ of dimension $(m,h)$.  Let $\ell \in \mathbb{N}$ and $0 \leq r < w$ be such that $m = \ell w +r$.  Let $R$ be any rectangle of dimension $(h,m)$ and divide $R$ into $\ell$ sub-rectangles, $R_1, R_2, \ldots, R_{\ell}$, each of height $h$ and width $w$,  with a remainder sub-rectangle of width $r$, denoted $R_0$.  

For each $i=0, 1, 2, \ldots, \ell$, it might not be the case that $R_i$ is horizontally traversable by $X$ since this event might depend on sites in adjacent sub-rectangles.  

Since membership in the set $X^+$ depends only on the initial infection of sites within distance $3$, it is at least true that the sub-rectangle of $R_i$ obtained by deleting $3$ columns from each side is horizontally traversable by $(X\cap R_i)^+$.  Denote these sub-rectangles by $R_0', R_1', \ldots, R_{\ell}'$.    Set 
\[
Q_1 = \max\{Q_1(B,Z,r-3), Q_1(B,Z, w-6)\}.
\]
Applying Lemma \ref{L:cross-upper} to the sub-rectangles $R_0', R_1', \ldots, R_{\ell}'$, 
\begin{align*}
\prob(R &\text{ is horiz. trav. by } X^+)\\
	&\leq \prod_{i=0}^{\ell}\prob(R_i' \text{ is horiz. trav. by } X^+)\\
	&\leq (1+pQ_1)^{\ell}e^{-g(hq^2)(w-7)\ell}(1+pQ_1)e^{-g(hq^2)(r-4)}\\
	&\leq (1+pQ_1)^{\ell+1}e^{-g(hq^2)(w \ell + r - 7\ell - 4)}\\
	&\leq (1+pQ_1)^{m/w+1} e^{-g(hq^2)(m-(7\ell+4))}\\
	&\leq (1+pQ_1)^{m/w+1} e^{-g(hq^2)m(1-11/w)}
\end{align*}  
yielding the desired bound on the probability that $R$ is traversable.
\end{proof}

As in Section \ref{S:lb}, consider the probability that infection spreads from a smaller rectangle to a larger one.  Previously, it was enough to examine only the probability that infection spreads from one particular square to another, but for the upper bound, a more general event is used.   

\begin{definition}
For any two rectangles $R \subseteq R'$ and $X\sim \operatorname{Bin}(R',p)$, let $D(R,R')$ be the event that $R'$ is internally spanned by $R \cup X$.
\end{definition}

  Essentially, this is the event that the four rectangles surrounding $R$ in $R'$ are traversable by the sites in $X \setminus R$.  The following lemma shows that even though these events are not independent, they are nearly so.

\begin{lemma}\label{L:D(R,R')}
For every $B\geq 1$, $Z \geq 0$ and $c \in (0, 1/6)$, there exist $T\geq 0$ and $p_1 = p_1(Z,c)$ such that for all $p \leq p_1$ and all $m$, $n$, $s$ and $t$ with $Z/p^2 \leq m,n \leq B/p^2$, and $s,t \leq T/p^2$ if $R \subset R'$ are two rectangles with dimensions $(m,n)$ and $(m+s, n+t)$, respectively, then 
\begin{multline*}
\prob(D(R,R')) \leq \\
	3(1+pQ_1(B,Z,\lceil 11/c \rceil))^{\frac{12}{11}c(s+t)+4} e^{16g(Z)-(1-6c)(sg(nq^2) + tg(mq^2))}.
\end{multline*}
\end{lemma}

\begin{proof}
Fix $B>1$, $Z>0$, $c \in (0, 1/6)$,  $p>0$ and let $R$ be a rectangle of dimension $(m,n)$ and let $R'$ be a rectangle of dimension $(m+s, n+t)$ with $R \subseteq R'$.  Suppose, without loss of generality that $s \leq t$.  Let $R' = [a_1, a_2] \times [b_1, b_2]$ and $R = [c_1, c_2] \times [d_1, d_2]$.  The rectangle $R'$ is decomposed into $R$ together with the following $8$ sub-rectangles,
\begin{align*}
R_1 &= [a_1, c_1-1] \times [b_1, d_1-1]		&R_2&=[c_1, c_2]\times [b_1, d_1-1]\\
R_3&=[c_2+1, a_2]\times [b_1, d_1-1]		&R_4&=[c_2+1, a_2]\times [d_1, d_2]\\
R_5&=[c_2+1, a_2] \times [d_2+1, b_2]		&R_6&=[c_1, c_2]\times [d_2+1, b_2]\\
R_7&=[a_1, c_1-1] \times [d_2 +1, b_2]		&R_8&=[a_1, c_1-1]\times [d_1, d_2].
\end{align*}


Let $X \sim \operatorname{Bin}(R', p)$.  If the event $D(R, R')$ occurs, then each of the rectangles $R_3 \cup R_4 \cup R_5$ and $R_7 \cup R_8 \cup R_1$ are horizontally traversable and each of the rectangles $R_5 \cup R_6 \cup R_7$ and $R_1 \cup R_2 \cup R_3$ are vertically traversable. The probability of each of these events can be individually approximated by Lemma \ref{L:cross-strips}, but these events are not independent.  Conditioning on the infected sites in the corner rectangles, $R_1, R_3, R_5$, and $R_7$, it is possible to approximate the probability of these events by slightly different events that are independent of each other.  

Set $Y=X^+ \cap (R_1 \cup R_3 \cup R_5 \cup R_7)$.  Since $|R_1 \cup R_3 \cup R_5 \cup R_7| =st$, then $\expect |Y| \leq st(8p^2 + 100p^3) \leq 9 st p^2$ for $p \leq 1/100$.  

The events that two particular sites are contained in $X^+$ are not independent, however, if $d(\mathbf{x}, \mathbf{y}) \geq 7$, then the events $\{\mathbf{x} \in X^+\}$ and $\{\mathbf{y} \in X^+\}$ are independent since they each depend on the initial infection of disjoint sets of sites.  

The grid, $\mathbb{Z}^2$, can be decomposed into $25$ disjoint sets $C_1, \ldots, C_{25}$ such that for each $i \in [1,25]$ and $\mathbf{x}, \mathbf{y} \in C_i$, $d(\mathbf{x}, \mathbf{y}) \geq 7$.   Indeed, set $B=B_5(\mathbf{0})$ and for each $\mathbf{b} \in B$, define $C_{\mathbf{b}} = \{\mathbf{b} + x(4,3)+y(3,-4):\ x, y \in \mathbb{Z}\}$.  These sets $\{C_{\mathbf{b}}:\ \mathbf{b} \in B\}$ are disjoint, $|B| = 25$ and for any $\mathbf{x}, \mathbf{y} \in C_{\mathbf{b}}$, if $\mathbf{x} \neq \mathbf{y}$, then $d(\mathbf{x}, \mathbf{y}) \geq 7$ and hence the events $\{\mathbf{x} \in X^+\}$ and $\{\mathbf{y} \in X^+\}$ are independent.

Now, if $|Y| \geq cs$, then for some $\mathbf{b} \in B$, the expected number of sites in $C_{\mathbf{b}} \cap Y$ satisfies $|C_{\mathbf{b}} \cap Y| \geq cs/25$.  For each $\mathbf{b} \in B$, $\expect(|C_{\mathbf{b}} \cap Y|) \leq 9p^2 st/25$ and thus by inequality \eqref{E:binomtail}, for $T \leq c/9$,
\begin{align*}
\prob \left(|C_{\mathbf{b}} \cap Y| \geq \frac{cs}{25}\right)
	&\leq \left(\frac{9p^2st/25}{cs/25} \right)^{cs/25}\\
	&=\left( \frac{9p^2 t}{c}\right)^{cs/25}\\
	&\leq \left( \frac{9T}{c}\right)^{cs/25}.
\end{align*}
Thus, 
\[
\prob(|Y| \geq cs) \leq 25\left(\frac{9T}{c} \right)^{cs/25}.
\]
 Choose $T = T(c,Z) \leq \frac{c}{9}\left(\frac{1}{25}e^{-2(1-6c)g(Z))} \right)^{25/c}$.  Since $g$ is a decreasing function,  $g(mq^2), g(nq^2) \leq g(Z)$ and hence since $s \leq t$, 
\begin{align*}
25 \left( \frac{9T}{c}\right)^{cs/25} 
	&\leq e^{-2s(1-6c)g(Z)}\\
	& \leq e^{-(1-6c)(sg(nq^2) + tg(mq^2))}\\
	& \leq e^{-(1-6c)sg(nq^2)}\leq 1.
\end{align*}

Similarly, for $s \leq T/p^2 \leq c/(9p^2)$,
\[
\prob(|Y| \geq ct) \leq  25 \left( \frac{9T}{c}\right)^{ct/25} \leq e^{-2t(1-6 c)g(Z)}.
\]

Consider the probability of the event $D(R, R')$ conditioning on $|Y| \leq cs$.  If every column of $R'$ that contained sites of $X^+ \cap (R_1 \cup R_3 \cup R_5 \cup R_7)$ were removed, the rectangles $R_4$ and $R_8$ would be split into as most $cs+2$ sub-rectangles of height $n$ and total width at least $s-cs$.

If $D(R,R')$ occurs, then in particular, each of these sub-rectangles is horizontally traversable by the sites in $X^+$.  However,  the membership of sites in $X^+$ might depend on initially infected sites in the deleted columns or adjacent rectangles.  In order to obtain a set of rectangles for which the events that each are horizontally traversable are independent, two further columns on either side of each sub-rectangle are removed.  Since this might also depend on sites in $Y$, delete $2$ further rows from the top and bottom of each sub-rectangle to ensure that the events are independent of the sites in $Y$.  Let the sub-rectangles be of widths $s_1, s_2, \ldots, s_j$ and note that $\sum_{i=1}^j s_i \geq s-cs-4(cs+2) = s(1-5c)-8$.  Set $w =\lceil 11/c \rceil$, let $Q_1 = Q_1(B,Z,w)$ and apply Lemma \ref{L:cross-strips} using $w$ for the widths of the strips.  Then by the choice of $w$, and since $g$ is decreasing,
\begin{align*}
\prob(R_4 &\text{ and } R_8 \text{ are horiz. trav.}\mid |Y|\leq cs)\\
	&\leq \prod_{i=1}^j(1+pQ_1)^{s_i/w+1}e^{-g((n-4)q^2)s_i(1-11/w)}\\
	&\leq (1+pQ_1)^{s/w + cs+2}e^{-(s(1-5c)-8)(1-11/w)g(nq^2)}\\
	&\leq (1+pQ_1)^{s/w+cs+2}e^{-g(nq^2)(s(1-6c)-8)}\\
	&\leq (1+pQ_1)^{\frac{12}{11}sc+2} e^{8g(Z) - g(nq^2)s(1-6c))}.
\end{align*}

Similarly, conditioning on the event that  $|Y| \leq ct$,
\[
\prob(R_2 \text{ and } R_6 \text{ are vert. trav.}\mid |Y|\leq ct) \leq  (1+pQ_1)^{\frac{12}{11}ct+2} e^{8g(Z) - t(1-6c)g(mq^2)}.
\]

Consider the event $D(R, R')$ conditioned on the following three possible ranges for the values of $|Y|$: $|Y| \leq cs$, $cs<|Y| \leq ct$, and $|Y| > ct$.
\begin{align*}
\prob(D(R,R')|&\ |Y| \leq cs)\prob(|Y| \leq cs)\\
	&\leq \prob(D(R,R')\mid |Y| \leq cs)\\
	&\leq (1+pQ_1)^{\frac{12}{11}c(s+t)+4}\\
	&\qquad \exp\left({16g(Z)-(1-6c)(tg(mq^2) + sg(nq^2))}\right),\\
\prob(D(R,R')|&\ cs < |Y| \leq ct)\prob(cs < |Y| \leq ct)\\
				&\leq \prob(D(R,R')\mid cs < |Y| \leq ct)\prob(cs < |Y|)\\
					&\leq (1+pQ_1)^{\frac{12}{11}ct+2} \exp\left({8g(Z) - t(1-6c)g(mq^2)}\right) e^{-(1-6c)sg(nq^2)}\\
				&\leq (1+pQ_1)^{\frac{12}{11}c(s+t)+4}\\
				&\qquad \exp\left({16g(Z) - (1-6c)(tg(mq^2) + sg(nq^2))}\right), \text{ and}\\
\prob(D(R,R')|&\ |Y| > ct) \prob(|Y|>ct)\\
		&\leq \prob(|Y|>ct)\\
		&\leq e^{-(1-6c)(sg(nq^2) + tg(mq^2))}\\
		&\leq (1+pQ_1)^{\frac{12}{11}c(s+t)+4}\\
		&\qquad \exp\left({16g(Z)-(1-6c)(sg(nq^2) + tg(mq^2))}\right).
\end{align*}
Combining these yields,
\begin{multline*}
\prob(D(R,R')) \leq 3(1+pQ_1)^{\frac{12}{11}c(s+t)+4}\\ \exp\left({16g(Z)-(1-6c)(sg(nq^2) + tg(mq^2))}\right),
\end{multline*}
the desired upper bound for the probability that the infection grows from the rectangle $R$ to the rectangle $R'$.
\end{proof}

\subsection{Hierarchies}

As in the works of Holroyd \cite{aH03} and also of Balogh and Bollob\'{a}s \cite{BB06}, the notion of a `hierarchy' is used to account for the different ways in which small internally spanned rectangles can either join together or grow into larger rectangles through the update process.  The definitions and results in this section are similar to the notion of hierarchies in \cite{aH03}, though on a different scale with respect to the parameter $p$ and with changes throughout to account for the behaviour of sites that become uninfected, but allow infection to spread between two larger infected rectangles.

\begin{definition}
A \emph{hierarchy for a rectangle $R$}, is a pair $\mathcal{H} = (G_{\mathcal{H}}, \{R_u\}_{u \in V(G_{\mathcal{H}})})$, where $G_{\mathcal{H}}$ is a finite directed rooted tree  with all edges directed away from the root and with maximum out-degree $3$, together with a collection of rectangles $\{R_u\}_{u \in V(G_{\mathcal{H}})}$ such that 
\begin{itemize}
\item if $r$ is the root of $G_{\mathcal{H}}$, then $R_r = R$,
\item if $u \to v$ in $G_{\mathcal{H}}$, then $R_u \supseteq R_v$,
\item if $u$ has three children, then at least one child has as its corresponding rectangle a single site,
\item if $u$ has two or three children and at least one child $v$ has $\text{short}(R_v) >2$, then $R_u$ is internally spanned by the rectangles corresponding to its children.
\end{itemize}
Vertices with out-degree $0$ are called \emph{seeds}, vertices with out-degree $1$ are called \emph{normal} and vertices with out-degree $2$ or $3$ are called \emph{splitters}. 
\end{definition} 

As in the analysis of usual bootstrap percolation, hierarchies are thought of as constructed `bottom up' using initially infected sites: two rectangles are joined to create a `parent' when their sites span a single larger rectangle.  There is a slight modification to deal with the case when one of these rectangles is a single site.  In this case, in order to remain consistent with the definition of $X^+$, a single site is only joined to another rectangle if the site is part of a triplet among the initially infected sites.  In this case, the rectangles joined will be those that correspond to sites in the triplet.

\begin{proposition}\label{P:finetree}
Let $R$ be a rectangle, $X \subseteq R$ and set $X^* = X \cap X^+$.  Suppose $R$ is internally spanned by $X$.  Then, there exists a hierarchy $\mathcal{H} = (G, \{R_u\}_{u\in V(G)})$ for $R$ and $\{X_u\}_{u \in V(G)}$ with $X_u \subseteq X^* \cap R_u$ such that
\begin{itemize}
	\item the root $r \in V(G)$ has $R_r = R$,
 	\item the rectangles corresponding to the seeds of $\mathcal{H}$ are all the individual sites in $X^*$, 
	\item every vertex that is not a seed has out degree at least $2$,
	\item if $u$ and $w$ are both children of a vertex $v$, then $X_u \cap X_w = \emptyset$. 
\end{itemize}
\end{proposition}

\begin{proof}
Note that by the definition of $X^+$, every site in the set $X^*$ is either part of a double or a triplet of sites in $X$ and the only sites in $X \setminus X^*$ are those that do not contribute to the final infection of any other sites before they recover.  Thus, $R$ is internally spanned by $X$ if{f} $R$ is internally spanned by $X^*$.

 The hierarchy $\mathcal{H}$ can be constructed recursively.  Let $R_1^{0}, R_2^{0}, \ldots, R_k^{0}$ be the individual sites in $X^*$ and let these correspond to the seeds of the hierarchy $\mathcal{H}$.  Given a partially constructed hierarchy $\mathcal{H}$, if there exist two vertices $u$ and $v$ with no parent so that $d(R_u, R_v) \leq 2$, add a new vertex to $G_{\mathcal{H}}$ by the following rules:

\textbf{Case 1:} If neither $R_u$ nor $R_v$ is a single site add a new vertex $w$ as the parent of $u$ and $v$ with $R_w$ the smallest rectangle that contains $R_u \cup R_v$ and set $X_w = X_u \cup X_v$.\\

\textbf{Case 2:} If $R_u = \mathbf{x}$ is a single site, then by the choice of $X^*$, the site $\mathbf{x}$ is part of either a double or a triplet.  The sites that form either the double or triplet containing $\mathbf{x}$ might already be a part of another rectangle, but in either case, there is either another rectangle $R_u'$ or two rectangles $R_u'$ and $R_u''$ with no parents that contain the sites associated with the double or triplet containing $\mathbf{x}$.  Add a new vertex $w$ as in the previous case and join either the rectangle and the site or the two rectangles and the site.

This process continues until there are no more sites or rectangles that have yet to be joined.  Since $R$ is internally spanned by $X^*$, this process will stop only when the last remaining vertex with no parent is a root that corresponds to the rectangle $R$.  The resulting directed graph and collection of rectangles have the desired properties, by induction.
\end{proof}

\begin{definition}
Given an initial infection $X$ of $R$, the hierarchy $\mathcal{H}$ is said to \emph{occur} (with respect to $X$) if{f} 
\begin{itemize}
	\item for every seed $u$, if the short side of $R_u$ is the horizontal side then in every $4$ adjacent columns in the the rectangle $R_u$, there are at least $2$ initially infected sites within distance $2$ (similarly for sets of $4$ adjacent rows if the short side of $R_u$ is vertical),
	\item  for every normal vertex $u$ with $u \to v \in E(G_{\mathcal{H}})$, the event $D(R_v, R_u)$ holds,  
\end{itemize}
and these events occur disjointly.

For any rectangle $R_u$ let $J(R_u)$ be the event, as above, that in every $4$ adjacent columns, there are at least $2$ initially infected sites within distance $2$ if the short side of $R_u$ is horizontal and similarly for set of $4$ adjacent rows if the short side of $R_u$ is vertical.
\end{definition}

The condition that these events occur disjointly is included so that by the van den Berg-Kesten inequality (inequality \eqref{E:vdBK-ineq}),
\[
\prob(\mathcal{H} \text{ occurs}) \leq \prod_{w \text{ seed}}\prob(J(R_w))\prod_{\underset{u \to v}{u \text{ normal}}} \prob(D(R_u, R_v)).
\]

Note that, for a rectangle $R$, the event that some hierarchy occurs is not equivalent to the event that the rectangle $R$ is internally spanned.  Rectangles corresponding to seeds might have two initially infected sites within distance two in every set of $4$ adjacent columns without being internally spanned.  However, as long as the rectangles corresponding to seeds are not too large, the difference will be small.  The definition is made in this way because, by Proposition \ref{P:finetree}, if $R$ is internally spanned by $X$, then there is a hierarchy $\mathcal{H}$ for $R$ that occurs.  The number of these hierarchies might be too large compared to the probability that a particular hierarchy occurs to give reasonable estimates on the probability that $R$ is internally spanned.  For this reason,   it is useful to consider the following types of hierarchies where the difference in dimensions between parent and child rectangles are not arbitrarily small.

\begin{definition}
Given $Z>T >0$ and $p>0$, the hierarchy $\mathcal{H}$ is said to be \emph{good} for $Z,T$, and $p$ if the rectangles $\{R_u\}_{u \in V}$ satisfy the following additional conditions on their dimensions:
\begin{itemize}
	\item if $v$ is a seed, then $\text{short}(R_v) < 2Z/p^2$,
	\item if $v$ is not a seed, then $\text{short}(R_v) \geq 2Z/p^2$,
	\item if $u$ is normal with child $v$, then $\phi(R_u) - \phi(R_v) \leq T/p^2$
	\item if $u$ is normal with $u \to v$ and $v$ is also normal, then $\phi(R_u) - \phi(R_v) \geq \frac{T}{2p^2}$, and
	\item if $u$ is a splitter and $v$ is a child of $u$, then $\phi(R_u) - \phi(R_v) \geq \frac{T}{2p^2}$.
\end{itemize}
\end{definition}

Next, it is shown that there exist hierarchies that are both good and occur for  rectangles that are internally spanned.

\begin{proposition}
Let $Z>T>0$, $p>0$ and let $R$ be a rectangle and let $X \subseteq R$.  If $R$ is internally spanned by $X$, then there exists a hierarchy $\mathcal{H}$ that is good for $Z$, $T$ and $p$ and that occurs.
\end{proposition}

\begin{proof}
The proof proceeds by induction on $R$.  If $\text{short}(R) < 2Z/p^2$, then take $G_{\mathcal{H}}$ to be a single isolated vertex $r$ and $R_r = R$.  If $R$ is internally spanned, then $\mathcal{H} = (G_{\mathcal{H}}, \{R_r\})$ is a good hierarchy that occurs.

Assume now that $\text{short}(R) \geq 2Z/p^2$, so then $\phi(R) \geq 4Z/p^2$.  Construct a sequence $R \supseteq R_1 \supseteq \ldots$ from Proposition \ref{P:finetree} going down the tree from the root, always talking $R_i$ to be the largest rectangle.  Let $m\geq 1$ be the smallest such that $\phi(R)-\phi(R_m) \geq \frac{T}{2p^2}$ and consider the following three cases.

\noindent \textbf{Case 1}:  If $\frac{T}{2p^2} \leq \phi(R)-\phi(R_m) \leq \frac{T}{p^2}$, then let $\mathcal{H}' = (G', \{R_u\}_{u \in V'})$ be a good hierarchy and denote the root by $r'$, corresponding to the rectangle $R_m$.  Let $r$ be a new vertex and define a new hierarchy rooted at $r$ with $R_r = R$ as follows.  Set $G = (V'\cup\{r\}, E(G') \cup \{r \to r'\})$ and then $\mathcal{H} = (G, \{R_u\}_{u \in V(G)})$ is the desired hierarchy.

\noindent \textbf{Case 2}: If $\phi(R)-\phi(R_m) > T/p^2$ and $m=1$, let $R_1'$ be the other rectangle from the tree in Proposition \ref{P:finetree}.  Note that by construction, $\phi(R_1') \leq \phi(R_1) \leq \phi(R) - \frac{T}{2p^2}$.  Let $\mathcal{H}_1$ and $\mathcal{H}_2$ be good hierarchies that occur for $R_1$ and $R_1'$, respectively.  Construct a good hierarchy for $R$ by adding a new vertex $r$ as the root, with edges joining it to the roots of the trees for $\mathcal{H}_1$ and $\mathcal{H}_1$.

\noindent \textbf{Case 3}: If $\phi(R)-\phi(R_m) > T/p^2$ and $m \geq 2$, let $R_m'$ be the other rectangle contained in $R_{m-1}$ from the tree in Proposition \ref{P:finetree}.  Let $\mathcal{H}_1$ and $\mathcal{H}_2$ be good hierarchies that occur for $R_{m}$ and $R_{m}'$ respectively.  For $i=1,2$, denote the root of $\mathcal{H}_i$ by $r_i$.  Let $r$ and $u$ be two new vertices and set $R_r = R$ and $R_u = R_{m-1}$.  Define a new hierarchy $\mathcal{H}$ with $G_{\mathcal{H}} = G_{\mathcal{H}_1} \cup G_{\mathcal{H}_2} \cup \{r \to u, u\to r_1, u \to r_2\}$, rooted at $r$.   Since $\phi(R) - \phi(R_{m-1}) < \frac{T}{2p^2}$ and $\phi(R) - \phi(R_m) \geq T/p^2$ then
\[
\phi(R_{m-1}) - \phi(R_m') \geq \phi(R_{m-1}) - \phi(R_m) \geq \frac{T}{p^2} - \frac{T}{2p^2} = \frac{T}{2p^2}.
\]
Hence, $\mathcal{H}$ is a good hierarchy for $T, Z$ and $p$.
\end{proof}

Good hierarchies are useful because there are not too many of them for rectangles of certain dimensions.  Fix $B \geq 1$, $p>0$ and let $R$ be a rectangle with $\text{short}(R) \leq \text{long}(R) \leq B/p^2$.  Let $Z,T>0$ and let $\mathcal{H}$ be a hierarchy for $R$ that is good for $Z, T$ and $p$.  By the definition of good hierarchies, for every directed path of length two in $G_{\mathcal{H}}$, $u \to v \to w$, the rectangles $R_u$ and $R_w$ satisfy $\phi(R_u) -\phi(R_w) \geq \frac{T}{2p^2}$.  Thus, the height of the tree $G_{\mathcal{H}}$ is at most
\[
2 \frac{2B/p^2}{T/(2p^2)}+1 = \frac{8B}{T} + 1.
\]
Since the out-degree of each vertex is at most $3$, there are at most $3^{8B/T+2}$ vertices in $G_{\mathcal{H}}$.  Set $M = M(B,T) = 3^{8B/T+2}$.

There are at most $M^{M-1}$ different rooted trees among all those belonging to a good hierarchy for $R$.  Consider now the number of different collections of rectangles corresponding to hierarchies.  In $R$, the number of different rectangles is 
\[
\binom{\text{long}(R)+1}{2}\binom{\text{short}(R)+1}{2} \leq \frac{(B/p^2+1)^4}{4} \leq \left(\frac{B}{p^2}\right)^4.
\]
Thus, for any rooted tree $G$ on at most $M$ vertices, there are at most $(B/p^2)^{4M}$ different collections $\{R_u\}_{u \in V(G)}$ such that for each $u \in V(G)$, $R_u$ is a rectangle contained in $R$.  Therefore, in total, there are at most
\begin{equation}\label{E:hierarchy_num}
M^{M-1}\left(\frac{B}{p^2} \right)^{4M} = M^{M-1}B^{4M}p^{-8M}
\end{equation}
different good hierarchies for the rectangle $R$.  While this number might be very large, it turns out to be small enough compared to the probability that a given hierarchy occurs to give a reasonable upper bound on the probability that the rectangle $R$ is internally spanned.

The following definitions and lemmas can be found in the paper by Holroyd \cite{aH03}.  Although, in that article, the function $g$ is different, the proofs use only the properties that the function $g$ is continuously differentiable, positive, decreasing and convex.  The function $g$, given by equation \eqref{E:betag}, has these properties, by definition and by Fact \ref{Fact:gconvex}.  To emphasize that the functions to come depend on $g$, we shall use $W_g$ and note that for all of these, the $g$ in question is that given in equation \eqref{E:betag}.
\begin{definition}
Let $\mathbf{a} = (a_1, a_2)$ and $\mathbf{b} = (b_1, b_2)$ with for $i=1,2$,  $0 \leq a_i \leq b_i$. Define
\[
W_g(\mathbf{a}, \mathbf{b}) = \inf\left\{\int_{\gamma} g(y)\ dx + g(x) dy\mid \gamma: \mathbf{a} \to \mathbf{b} \text{ piecewise linear path}\right\}.
\]
\end{definition}

The function $W$ and its properties are used to bound the term $\exp(tg(mq^2)+sg(nq^2))$ arising in Lemma \ref{L:D(R,R')}. The following, Lemmas \ref{L:Wtriangleineq}, \ref{L:oneline}, \ref{L:Wint}, \ref{L:Wpod}, are from Holroyd \cite{aH03} (Propositions 12, 13, 14, and 15).

\begin{lemma}\label{L:Wtriangleineq}
Let $\mathbf{a}, \mathbf{b}, \mathbf{c} \in (\mathbb{R}^+)^2$ with $\mathbf{a} \leq \mathbf{b} \leq \mathbf{c}$.  Then 
\[
W_g(\mathbf{a}, \mathbf{b}) + W_g(\mathbf{b}, \mathbf{c}) \geq W_g(\mathbf{a}, \mathbf{c}).
\]
\end{lemma}

\begin{lemma}\label{L:oneline}
If $\mathbf{a} \leq \mathbf{b}$, then $W_g(\mathbf{a}, \mathbf{b}) \leq (b_1-a_1)g(a_2) + (b_2-a_2)g(a_1)$.
\end{lemma}

\begin{lemma}\label{L:Wint}
If $\mathbf{a} = (a_1, a_2)$ with $a_1+a_2 = A$ and $\mathbf{b} = (B,B)$ with $A \leq B$, then
\[
W_g(\mathbf{a}, \mathbf{b}) \geq 2 \int_A^B g(x)\ dx.
\]
\end{lemma}

\begin{lemma}\label{L:Wpod}
For every $z, Z$ with $0 < z \leq Z$ and $\mathbf{a}, \mathbf{b}, \mathbf{c}, \mathbf{d}, \mathbf{r} \in (\mathbb{R}^+)^2$ with $\mathbf{a} \leq \mathbf{b}$, $\mathbf{c} \leq \mathbf{d}$, $\mathbf{r} \geq \mathbf{b}, \mathbf{d}, (2Z, 2Z)$ and $\mathbf{r} \leq \mathbf{b} + \mathbf{d} + (a,a)$, there exists $\mathbf{s} \in (\mathbb{R}^+)^2$ with $\mathbf{s} \leq \mathbf{r}$ and $\mathbf{s} \leq \mathbf{a} + \mathbf{c}$ such that
\[
W_g(\mathbf{a}, \mathbf{b}) + W_g(\mathbf{c}, \mathbf{d}) \geq W_g(\mathbf{s}, \mathbf{r}) - 2zg(Z).
\]
\end{lemma}

\begin{definition}
For rectangles $R \subseteq R'$, set
\[
U(R,R') = W_g(q^2\dim(R), q^2\dim(R')).
\]
\end{definition}

This definition is useful since if $R$ is a rectangle of dimension $(m,n)$ and $R'$ is a rectangle of dimension $(m+s, n+t)$ with $R \subseteq R'$, then by Lemma \ref{L:oneline}, 
\[
\frac{U(R,R')}{q^2} \leq sg(mq^2)+tg(nq^2).
\]

The following lemma, adapted from a corresponding result in \cite{aH03}, shows that every hierarchy is associated with a rectangle called a `pod' that can be used to bound the sum of the values $U(R_v, R_u)$ over all normal vertices in the hierarchy.

\begin{lemma}\label{L:pods}
Fix $Z,T, q$ with $3q^2 < Z$, let $\mathcal{H}$ be a good hierarchy for the rectangle $R$ with root $r$ and let $N_s(\mathcal{H})$ be the number of vertices in $G_{\mathcal{H}}$ that are splitters.  There exists a rectangle $S = S(\mathcal{H})$ with $S \subseteq R$ and 
\[
\dim(S) \leq \sum_{w \text{ seed}} \dim(R_w),
\]
and
\[
\sum_{\underset{u \text{ normal}}{u \to v}} U(R_v, R_u) \geq U(S,R) - 6N_s(\mathcal{H})q^2 g(Z).
\]
\end{lemma}

\begin{proof}
The proof proceeds by induction on the number of vertices in $G_{\mathcal{H}}$.  If $|V(G_{\mathcal{H}})| = 1$, then take $S=R$.

If $|V(G_{\mathcal{H}})|>1$, consider separately the cases where the root $r$ is a normal vertex or a splitter.  If $r$ is a normal vertex with child $u$, let $\mathcal{H}'$ be the sub-hierarchy with root $u$ and apply the induction hypothesis to $\mathcal{H}'$ to get a rectangle $S' \subseteq R_u$.  The hierarchy $\mathcal{H}'$ has the same number of splitters as the hierarchy $\mathcal{H}$ and the same seeds.  Thus
$\dim(S') \leq \sum_{w \in V(G_{\mathcal{H}'}) \text{ seed}} \dim(R_w)$ and
\begin{align*}
\sum_{v \to w} U(R_w, R_v)
	&\geq U(S, R_u)+ U(R_u, R)  - 6N_s(\mathcal{H}')q^2g(Z)\\
	&\geq U(S,R) -  6N_s(\mathcal{H})q^2g(Z).		&&\hspace{-20pt}\text{(by Lemma \ref{L:Wtriangleineq})}
\end{align*}

Suppose now that $r$ is a splitter and let $u$ and $v$ be two children of $r$ that correspond to the two largest rectangles among the children of $r$, disregarding a third site that corresponds to a single site. Let $\mathcal{H}_1$ and $\mathcal{H}_2$ be the two sub-hierarchies with roots $u$ and $v$ respectively.  Then, since $r$ is a splitter, $N_s(\mathcal{H}) = N_s(\mathcal{H}_1)+N_s(\mathcal{H}_2)+1$.  Also, $\dim(R_u)+\dim(R_v) \geq \dim(R) - (3,3)$, accounting for the case when there is a third vertex that corresponds to a site in a triplet.

Let $S_1 \subseteq R_u$ and $S_2 \subseteq R_v$ be given by the induction hypothesis and for $i=1,2$ set $\mathbf{s}_i = q^2\dim(S_i)$, $\mathbf{r}_1 = q^2\dim(R_u)$, $\mathbf{r}_2 = q^2\dim(R_v)$ and $\mathbf{r} = q^2\dim(R)$.  By Lemma \ref{L:Wpod}, there exists $\mathbf{s} \leq \mathbf{r}$ with $\mathbf{s} \leq \mathbf{s}_1+ \mathbf{s}_2$ such that
\[
W_g(\mathbf{s}_1, \mathbf{r}_1)+ W_g(\mathbf{s}_2, \mathbf{r}_2) \geq W_g(\mathbf{s}, \mathbf{r}) - 2(3q^2)g(Z).
\] 
Let $S$ be a rectangle in $R$ of dimension $\frac{1}{q^2}\mathbf{s}$.  Then 
\begin{align*}
\dim(S) 	&\leq \dim(S_1) + \dim(S_2)\\
		& \leq \sum_{w \text{ seed in } \mathcal{H}_1}\dim(R_w) + \sum_{w \text{ seed in } \mathcal{H}_2}\dim(R_w)\\
		& = \sum_{w \text{ seed in } \mathcal{H}}\dim(R_w).   
\end{align*}
Also, by the choice of $\mathbf{s}$, $U(S_1, R_1) + U(S_2, R_2) \geq U(S,R) - 6q^2g(Z)$ and by the choice of $S_1$ and $S_2$,
\begin{align*}
\sum_{\underset{x \to y}{x,y \in V(\mathcal{H})}} &U(R_y, R_x)\\
	& = \sum_{\underset{x \to y}{x,y \in V(\mathcal{H}_1)}} U(R_y, R_x) + \sum_{\underset{x \to y}{x,y \in V(\mathcal{H}_1)}} U(R_y, R_x)\\
	&\geq U(S_1, R_u) - 6q^2 N_s(\mathcal{H}_1) g(Z) + U(S_2, R_v) - 6q^2 N_s(\mathcal{H}_1) g(Z)\\
	&\geq U(S, R) - 6q^2g(Z)- 6q^2 N_s(\mathcal{H}_1) g(Z)-6q^2 N_s(\mathcal{H}_2) g(Z)\\
	&= U(S,R) - (N_s(\mathcal{H}_1)+ N_s(\mathcal{H}_2)+1)6q^2 g(Z)\\
	&=U(S,R) - 6q^2N_s(\mathcal{H})g(Z).
\end{align*}
By induction, the result holds for all good hierarchies, $\mathcal{H}$.
\end{proof}

Using the notion of pods, the following upper bound is given on the probability that squares of a particular size are internally spanned.  Recall from Section \ref{S:2tiles} that $\lambda_B$ is used to denote $\int_{1/B}^{B}g(x)\ dx \leq \lambda$.

\begin{theorem}\label{T:ubspan}
For every $\varepsilon >0$, there is a $B_0 = B_0(\varepsilon)>0$ such that for $B>B_0$ there exists $p_0 = p_0(\varepsilon, B)$ such that if $0< p< p_0$ then
\[
I(\lfloor B/p^2 \rfloor, p) \leq \exp\left(\frac{\varepsilon - 2(1-7/B)\lambda_B}{p^2} \right).
\]
\end{theorem}

\begin{proof}
Fix $\varepsilon>0$,  $B>1$ and $p>0$.  Set $n = \lfloor B/p^2 \rfloor$,  $c=1/B$, fix $Z>0$ and let $T$ be given by Lemma \ref{L:D(R,R')}.  By the van den Berg-Kesten inequality (inequality \eqref{E:vdBK-ineq}), for any hierarchy $\mathcal{H}$ that is good for $Z$, $T$ and $p$ with respect to $[n]^2$, 
\[
\prob(\mathcal{H} \text{ occurs}) \leq \prod_{w \text{ seed}}\prob(J(R_w))\prod_{\underset{u \to v}{u \text{ normal}}} \prob(D(R_u, R_v)).
\]

Consider first the terms $\prob(D(R_v, R_u))$. Set $Q_1 = Q_1(B,Z,\lceil 11/c \rceil)$.  By Lemma \ref{L:D(R,R')}, for a normal vertex $u$ and $u \to v$ with $\dim(R_u) = (m+s, n+t)$ and $\dim(R_v) =(m,n)$,
\begin{align*}
\prob(D&(R_v, R_u))\\
	 &\leq 3(1+pQ_1)^{\frac{12}{11}c(s+t)+4} e^{16g(Z)-(1-6c)(sg(nq^2) + tg(mq^2))}\\
	&\leq 3(1+pQ_1)^{\frac{24}{11}cT/p^2 +4}\exp\left(16g(Z) - (1-6c)\frac{U(R_v, R_u)}{q^2}\right)\\
	&\leq \exp\left(\log 3 +\left(\frac{24cT}{11p^2} +4 \right)pQ_1 + 16g(Z) -(1-6c)\frac{U(R_v, R_u)}{q^2} \right)\\
	&\leq \exp\left(\frac{(\log 3 + (24cT/11+4)Q_1 + 16g(Z))}{p} - \frac{(1-6c)U(R_v, R_u) }{q^2} \right).
\end{align*}
Set $Q_2 = Q_2(B, Z) = \log 3 + (24cT/11+4)Q_1 + 16g(Z)$.  Then, $Q_2$ is a constant that depends only $B$ and $Z$, since $T$ and $c$ depend only on $B$ and $Z$ and
\begin{equation}\label{E:growingR}
\prob(D(R_v, R_u)) \leq \exp(Q_2/p)\exp\left(-(1-6c)\frac{U(R_u, R_v)}{q^2}\right).
\end{equation}
Let $N_1(\mathcal{H})$ be the number of normal vertices in $G_{\mathcal{H}}$ and let $N_0(\mathcal{H})$ be the number of seeds.  Recall that the number of vertices in the hierarchy $\mathcal{H}$ is at most $M = 3^{8B/T +2}$ and so $N_1(\mathcal{H})$ and $N_0(\mathcal{H})$ are both at most a constant that depends only on $B$ and $Z$, since $T$ depends on $B$ and $Z$.
Let $S$ be a pod rectangle for $\mathcal{H}$ given by Lemma \ref{L:pods}.  Then, $\dim(S) \leq \sum_{w \text{ seed}} \dim(R_w)$ and 
\[
\sum_{\underset{u \to v}{u \text{ normal}}} U(R_v, R_u) \geq U(S,R) - 6N_s(\mathcal{H}) q^2 g(Z).
\]
Combining this with inequality \eqref{E:growingR},
\begin{align*}
\prod_{\underset{u \to v}{u \text{ normal}}} &\prob(D(R_u, R_v))\\
	&\leq \prod_{\underset{u \to v}{u \text{ normal}}} \exp\left(\frac{Q_2}{p}  - \frac{(1-6c)U(R_u, R_v)}{q^2}\right)\\
	&=\exp\left(\frac{N_1(\mathcal{H})Q_2}{p} - \frac{(1-6c)}{q^2} \sum_{\underset{u \to v}{u \text{ normal}}} U(R_u, R_v)\right)\\
	&\leq \exp\left(\frac{N_1(\mathcal{H})Q_2}{p} - \frac{(1-6c)}{q^2}\left(U(S,R) - 6N_s(\mathcal{H})q^2g(Z) \right) \right)\\
	&\leq \exp\left(\frac{N_1(\mathcal{H})Q_2p + 6N_s(\mathcal{H})g(Z)q^2}{p^2} - (1-6c)\frac{U(S,R)}{q^2}\right).
\end{align*}
Let $p$ be small enough so that $N_1(\mathcal{H})Q_2p + 6N_s(\mathcal{H})g(Z)q^2 \leq \varepsilon/3$, then
\begin{equation}\label{E:growingsimplified}
\prod_{\underset{u \to v}{u \text{ normal}}} \prob(D(R_u, R_v)) \leq \exp\left(\frac{\varepsilon/3}{p^2} - (1-6c)\frac{U(S,R)}{q^2}\right).
\end{equation}

To estimate the probability of the events $J(R_w)$, suppose without loss of generality that the short side of $R_w$ is horizontal and consider $\lfloor \text{long}(R_w)/4 \rfloor$ disjoint sets of $4$ adjacent columns in $R_w$.  In one set of $4$ adjacent columns, each site has at most $11$ sites within distance $2$ and so there are at most $\frac{11 \cdot 4}{2}\ \text{short}(R_w) \leq 22 Z/p^2$ pairs of sites within distance $2$.  The probability that at least one of these pairs has both sites initially infected is at most $22Z$ and hence 
\[
\prob(J(R_w)) \leq (22Z)^{\lfloor \text{long}(R_w)/4\rfloor} \leq (22Z)^{\phi(R_w)/8 - 1}.
\]
Thus, using the fact that $\dim(S) \leq \sum_{w \text{ seed}} \dim(R_w)$,
\begin{align*}
\prod_{w \text{ seed}} \prob(J(R_w))
	&\leq (22Z)^{\sum_{w \text{ seed}} \phi(R_w)/8-3/4}\\
	&\leq (22Z)^{\phi(S)/8 - N_0(\mathcal{H})}\\
	&= \exp \left(\frac{\log(22Z) \phi(S)}{8} - \log(22Z) N_0(\mathcal{H})\right).
\end{align*}
Let $p$ be small enough so that $-\log(22Z) N_0(\mathcal{H}) \leq \frac{\varepsilon}{3p^2}$.  Then
\[
\prod_{w \text{ seed}} \prob(J(R_w)) \leq \exp\left(\frac{\log(22Z) \phi(S)}{8}+ \frac{\varepsilon/3}{p^2} \right)
\]
and so
\begin{align*}
\prob(\mathcal{H} \text{ occurs}) 
	&\leq \exp\left(\frac{\log(22Z) \phi(S)}{8}+ \frac{\varepsilon/3}{p^2} + \frac{\varepsilon/3}{p^2} - (1-6c)\frac{U(S,R)}{q^2} \right)\\
	&=\exp\left(\frac{2 \varepsilon/3}{p^2} + \frac{\log(22Z) \phi(S)}{8} -\frac{(1-6c)U(S,R)}{q^2} \right).
\end{align*}

Consider two different cases, depending on the size of the semi-perimeter of the rectangle $S$.  

\noindent \textbf{Case 1}:  If $\phi(S) \leq \frac{1}{Bq^2}$, then applying Lemma \ref{L:Wint} with $q^2\phi(S) = A \leq 1/B$,
\[
U(S,R) = W_g(q^2\dim(S), q^2\dim(R)) \geq 2\int_{1/B}^B g(x)\ dx = 2\lambda_B
\]
and so 
\[
\exp \left(-\frac{(1-6c)}{q^2}U(S,R)\right) \leq \exp\left(-\frac{2(1-6c)}{q^2}\lambda_B\right).
\]
In this case,
\begin{align*}
\prob(\mathcal{H} \text{ occurs})
	&\leq \exp\left(\frac{2 \varepsilon/3}{p^2}  -\frac{(1-6c)U(S,R)}{q^2} \right)\\
	&\leq \exp\left(\frac{2\varepsilon/3}{p^2} -\frac{2(1-6c)\lambda_B}{q^2}\right).
\end{align*}

\noindent \textbf{Case 2}:  If, on the other hand, $\phi(S) > \frac{1}{Bq^2}$, then
\[
\frac{-\log(22Z)\phi(S)}{8} \geq \frac{-\log(22Z)}{8Bq^2}.
\]
Choose $Z>0$ to be small enough so that
\[
\frac{-\log(22Z)}{8B} \geq 2\lambda \geq 2(1-6c) \lambda_B.
\]
Then,
\begin{align*}
\prob(\mathcal{H} \text{ occurs})
	& \leq \exp\left(\frac{2 \varepsilon/3}{p^2} + \frac{\log(22Z) \phi(S)}{8} \right)\\
	& \leq \exp\left(\frac{2\varepsilon/3}{p^2} -\frac{2(1-6c)\lambda_B}{q^2}\right).
\end{align*}


Finally, recall that for $M = 3^{8B/T+2}$, there are at most $M^{M-1}B^{4M}p^{-8M}$ different good hierarchies $\mathcal{H}$.  Let $p$ be small enough so that $M^{M-1}B^{4M}p^{-8M} \leq \exp\left(\frac{\varepsilon/3}{p^2}\right)$.  Then,
\begin{align*}
I(L, p) 
	&\leq \sum_{\mathcal{H}} \prob(\mathcal{H} \text{ occurs})\\
	&\leq M^{M-1}B^{4M}p^{-8M} \exp\left(\frac{2\varepsilon/3}{p^2} -\frac{2(1-6c)\lambda_B}{q^2}\right)\\
	&=\exp\left(\frac{\varepsilon/3}{p^2} + \frac{2\varepsilon/3}{p^2}- \frac{2(1-6c)\lambda_B}{q^2} \right)\\
	&=\exp\left(\frac{\varepsilon}{p^2} - \frac{2(1-6c)\lambda_B}{q^2} \right).
\end{align*}
Let $p$ be small enough so that $\frac{q^2}{p^2} \leq \frac{1-6c}{1-7c}$.  Then,
\[
I(L,p) \leq \exp\left(\frac{\varepsilon - 2(1-7c)\lambda_B}{p^2} \right).
\]
\end{proof}

In order to extend Theorem \ref{T:ubspan} to give an upper bound on the probability that an arbitrarily large rectangle percolates, the following lemma is used.  If a large rectangle $R$ is internally spanned, it might not be possible to guarantee that $R$ will contain internally spanned squares of a particular scale, but the following shows that it is at least possible to guarantee the existence of internally spanned rectangles of a particular scale.  Lemma \ref{L:smallerspan} is an immediate analogue to a result on usual bootstrap percolation given in \cite{AL88}.

\begin{lemma}\label{L:smallerspan}
Fix a rectangle $R$, $k \in \mathbb{N}$  with $\text{long}(R) \geq 2k$, and $X_0 \subseteq R$.   If $R$ is internally spanned by $X_0$, then there exists a rectangle $T \subseteq R$ with $\text{long}(T) \in [k, 2k]$ that is internally spanned by $X_0$.
\end{lemma}

While Theorem \ref{T:ubspan} gives an upper bound on the probability of percolation for any large enough rectangle and small enough probability of initial infection, it remains to show how this can be used to give a bound on the critical probability.


\begin{theorem}\label{T:ublarge}
For every $\varepsilon >0$, there exists $n_0 = n(\varepsilon)$ such that for every $n \geq n_0$, if $p<0$ is such that $p \leq \sqrt{\frac{\lambda-\varepsilon}{\log n}}$, then
\[
I(n, p) \leq n^{-\frac{\varepsilon}{2(\lambda-\varepsilon)}}.
\]
\end{theorem} 

\begin{proof}
Fix $\varepsilon>0$ and let $B = B(\varepsilon)$ and $p_0 = p_0(\varepsilon)$ be given by Theorem \ref{T:ubspan} and with $B$ large enough so that $\lambda-(1-7/B)\lambda_B< \varepsilon/12$.  Let $n_0 = n_0(B, \varepsilon)$ be large enough so that $\sqrt{\frac{\lambda-\varepsilon}{\log n_0}}<p_0$ and if $n \geq n_0$, then  $n \geq \frac{B\log n}{\lambda-\varepsilon}$.

Fix $n > n_0$ and $p>0$ with $p \leq \sqrt{\frac{\lambda-\varepsilon}{\log n}}$.  Note that if $p \leq p'$, then by coupling, $I(n,p) \leq I(n,p')$ and so it suffices to prove the result assuming that $p =\sqrt{\frac{\lambda-\varepsilon}{\log n}}$.  By the choice of $n_0$, $n > \frac{B \log n}{\lambda-\varepsilon} = B/p^2$.  Set $R = [n]^2$.

Set $K = \lfloor B/p^2 \rfloor$ and $k = \lfloor B/2p^2 \rfloor$ so that $2k \leq K < n$.  By Lemma \ref{L:smallerspan}, if $R$ is internally spanned, then there is an internally spanned rectangle $T \subseteq R$ with $\text{long}(T) \in [k, 2k]$.  Thus,
\[
I(n, p) \leq \sum_{\underset{\text{long}(T) \in [k,2k]}{T \subseteq R}} I(T,p).
\]
In $R$, there are at most $n^2(2k)^2 \leq n^2K^2$ such rectangles $T$.  By the choice of $K$, and for $n$ sufficiently large,
\[
K^2 \leq \frac{B^2}{p^4} = \frac{B^2(\log n)^2}{(\lambda-\varepsilon)^2} \leq n^{\frac{\varepsilon}{6(\lambda-\varepsilon)}}.
\]

It remains to determine an upper bound on the probability of such a rectangle being internally spanned.  Fix such a rectangle $T$ of dimension $(a,b)$ and suppose without loss of generality that $a \leq b$ and that $T = [1,a]\times [1,b]$.  Consider one particular way in which the rectangle $[1,K]^2$ can be internally spanned.  The rectangle $[K]^2$ is internally spanned if $T$ is internally spanned and  every column of the rectangle $[a+1,K]\times [K]$ contains two adjacent initially infected sites and every row of the rectangle $[a]\times[b+1,K]$ contains two adjacent initially infected sites.  Since these events are all independent,
\begin{align*}
I(K,p)	&\geq I(T,p)(1-(1-p^2)^{\lfloor K/2\rfloor})^{K-b}(1-(1-p^2)^{\lfloor b/2 \rfloor})^{K-a}\\
				&\geq I(T,p)(1-e^{-p^2(K-1)/2})^{K}(1-e^{-p^2(k-1)/2})^{K}\\
				&\geq I(T,p) (1-e^{-p^2(k-1)/2})^{2K}\\
				&\geq I(T,p) (1-e^{-(B/4-1)})^{2K}\\
				&\geq I(T,p) \exp(-4Ke^{-B/4+1})		&&\text{(for $B\geq 5$)}\\
				&\geq I(T,p) \exp\left(\frac{-4Be^{-B/4+1}}{p^2}\right).
\end{align*}
Hence for any $T \subseteq R$ with $\text{long}(T) \in [k,2k]$, by Theorem \ref{T:ubspan} applied to $[K]^2$,
\[
I(T,p) 
	 \leq \exp\left(\frac{4Be^{-B/4+1}}{p^2}+ \frac{\varepsilon - 2(1-7/B)\lambda_B}{p^2}\right).
\]
Let $B$ be large enough so that $4B e^{-B/4+1} \leq \varepsilon/6$.  Since $(1-7/B)\lambda_B > \lambda - \frac{\varepsilon}{12}$,
\begin{align*}
I(T,p) 
	&\leq \exp\left(\frac{\varepsilon/6 + \varepsilon - 2(\lambda - \varepsilon/(12))}{p^2}\right)\\
	&=\exp\left(\frac{4\varepsilon/3 - 2\lambda}{p^2}\right)\\
	&=\exp\left(\frac{(4\varepsilon/3 - 2\lambda)\log n}{\lambda-\varepsilon}\right)\\
	&=n^{\frac{-(2\lambda - 4\varepsilon/3)}{\lambda-\varepsilon}} = n^{-2-\frac{2\varepsilon/3}{\lambda-\varepsilon}}.
\end{align*}
Therefore, the probability that $[n]^2$ is internally spanned can be bounded above as
\begin{align*}
I(n, p)	&\leq n^2 K^2 n^{-(2+2\varepsilon/3)}\\
		&\leq n^2 n^{\frac{\varepsilon}{6(\lambda-\varepsilon)}} n^{-2-\frac{2\varepsilon/3}{\lambda-\varepsilon}}\\
		&= n^{\frac{-\varepsilon}{2(\lambda-\varepsilon)}}
\end{align*}
as claimed.
\end{proof}

In particular, by Theorem \ref{T:ublarge} for each $\varepsilon>0$ and sequence $\{p(n)\}_{n \in \mathbb{N}}$ with the property that for each $n \in \mathbb{N}$,
\[
p(n) \leq \sqrt{\frac{\lambda-\varepsilon}{\log n}}
\]
then $I(n,p(n)) = o(1)$.

This implies that for every $\varepsilon>0$, there is an $n_{\varepsilon}$ such that for all $n \geq n_{\varepsilon}$,
\[
p_c([n]^2, \mathcal{R}) \geq \sqrt{\frac{\lambda - \varepsilon}{\log n}}.
\]
Combining Theorems \ref{T:ublarge} and \ref{T:lb}, this shows that the critical probability for the update rule $\mathcal{R}$ satisfies
\[
p_c([n]^2, \mathcal{R}) = \sqrt{\frac{\lambda + o(1)}{\log n}}.
\]

A remaining open problem would be to determine further terms in the expansion of the critical probability $p_c([n]^2, \mathcal{R})$ and to investigate variations of this process in other dimensions or on graphs other than the integer lattices.

\end{document}